\documentclass{amsart}

\newcommand{\HH}{\mathbb{H}}

\newcommand{\II}{\mathbb{I}}

\newcommand{\e}{\mathbf{e}}
\newcommand{\RR}{\mathbb{R}}

\newcommand{\Hc}{\mathbb{H}_{\text{cyc}}}
\newcommand{\Ic}{I_{\text{cyc}}}
\newcommand{\Jc}{J_{\text{cyc}}}
\newcommand{\Hr}{\mathbb{H}_{\text{reg}}}
\DeclareMathOperator{\cyc}{cyc}
\DeclareMathOperator{\reg}{reg}

\usepackage[colorlinks=true,allcolors=blue,final]{hyperref}
\hypersetup{
    hypertexnames=false, 
}
\usepackage[capitalize,nameinlink,noabbrev]{cleveref}

\usepackage{aliascnt}

\newtheorem{theorem}{Theorem}[section]

\newtheorem{thm}{Theorem}

\newaliascnt{lemma}{theorem}
\newtheorem{lemma}[lemma]{Lemma}
\aliascntresetthe{lemma}

\newaliascnt{remark}{theorem}
\newtheorem{remark}[remark]{Remark}
\aliascntresetthe{remark}

\newtheorem*{remark*}{Remark}

\newaliascnt{corollary}{theorem}

\aliascntresetthe{corollary}

\newaliascnt{proposition}{theorem}
\newtheorem{proposition}[proposition]{Proposition}
\aliascntresetthe{proposition}

\newaliascnt{conjecture}{theorem}

\aliascntresetthe{conjecture}

\theoremstyle{definition}

\newaliascnt{definition}{theorem}
\newtheorem{definition}[definition]{Definition}
\aliascntresetthe{definition}

\newaliascnt{example}{theorem}
\newtheorem{example}[example]{Example}
\aliascntresetthe{example}

\newaliascnt{condition}{theorem}

\aliascntresetthe{condition}

\newtheorem{claim}{Claim}[theorem]

\newcommand{\defn}[1]{\emph{\textcolor{gray}{#1}}}

\usepackage[margin=3.3cm]{geometry}
\usepackage[all]{xy}
\usepackage[dvipsnames]{xcolor}
\usepackage{adjustbox,algorithm,algpseudocode,alphalph,amsmath,amssymb,caption,color,soul,dsfont,enumerate,float,multirow,graphicx,subcaption,mathrsfs,natbib,pgfplots,ragged2e,tikz,tikz-cd,tkz-graph,todonotes,ulem,thmtools,mathtools,pifont}
\algrenewcommand\algorithmicensure{\textbf{Output:}}
\algrenewcommand\algorithmicrequire{\textbf{Input:}}
\usetikzlibrary{arrows.meta,bending,graphs,arrows,backgrounds,external,intersections,decorations.markings,shapes.geometric,tikzmark}
\pgfplotsset{compat=newest}
    \usetikzlibrary{intersections,pgfplots.fillbetween}
    \usepgfplotslibrary{fillbetween}
    \pgfdeclarelayer{pre main}

\title[Lattice characterization of cyclic interval hypergraphic posets]{Lattice characterization of cyclic interval hypergraphic posets}

\author[F.~Gélinas]{Félix Gélinas}
\address[F.~Gélinas]{York University, Toronto}
\email{felixgel@yorku.ca}
\urladdr{https://felixgelinas.github.io/}

\author[Y.~Yang]{Yirong Yang}
\address[Y.~Yang]{University of Washington, Seattle}
\email{yyang1@uw.edu}
\urladdr{https://sites.google.com/view/yirongyang/}

\begin{document}

\begin{abstract}
   Hypergraphic polytopes $\Delta_\HH$ arise as Minkowski sums of simplices indexed by the hyperedges of a hypergraph $\HH$. Orienting the $1$-skeleton of such a polytope by a certain generic linear functional gives rise to the hypergraphic poset $P_\HH$. Hypergraphic posets include the weak order for the permutahedron and the Tamari lattice for the associahedron. This motivates the problem of determining when $P_\HH$ is a lattice. In this paper, we give a complete lattice characterization for cyclic interval hypergraphs, extending the result of Bergeron and Pilaud for interval hypergraphs, and the result of Adenbaum et al. for the complete cyclic interval hypergraph.
\end{abstract}

\maketitle
\normalem

\tableofcontents

\section{Introduction}\label{sec:intro}

Let $(\e_i)_{i \in [n]}$ denote the standard basis of $\RR^n$ for $n \ge 1$. For a hypergraph $\HH$ on $[n]:=\{1, \dots, n\}$, define its \emph{hypergraphic polytope} $\Delta_\HH$ to be the Minkowski sum $\sum_{H \in \HH} \Delta_{H}$, where $\Delta_{H}$ is the $(|H|-1)$-dimensional simplex on vertices $\{\e_{h} \ | \ h \in H\}$. Hypergraphic polytopes were introduced simultaneously in \cite{BenedettiBergeronMachacek2019} and \cite{AguiarArdila2023} and have been studied extensively (see for example~\cite{FeichtnerSturmfels2005, Postnikov2009, PostnikovReinerWilliams2008, AgnarssonMorris2009, Agnarsson2017, AguiarArdila2023, BenedettiBergeronMachacek2019, PadrolPilaudPoullot2022, Rehberg2022, CardinalSteiner2025, CardinalHoangMerinoMickaMutze2023, BergeronPilaud2026, AbramBastidasGelinasPilaudSack2025, Gelinas2025}). Notably, $\Delta_\HH$ is the classical permutahedron when $\HH$ is the complete graph ${[n] \choose 2}$, the associahedron~\cite{ShniderSternberg1993, Loday2004} when $\HH$ is the complete interval hypergraph $\{[i, j] \ | \ 1 \le i < j \le n\}$, a graphical zonotope when $\HH \subseteq {[n] \choose 2}$, the nestohedra~\cite{FeichtnerSturmfels2005, Postnikov2009} when $\HH$ is a building set, and the cyclohedron~\cite{bott1994self} when $\HH$ is the complete cyclic interval hypergraph.

Given a polytope, it is natural to study the directed acyclic graph obtained from its $1$-skeleton by orienting each edge according to a generic linear functional. For example, the acyclic orientation corresponds to certain shellings of the dual polytope, and is a key ingredient in Kalai's proof~\cite{Kalai1988} that all simple polytopes are combinatorially determined by their $1$-skeleta. For a discrete Morse theoretic interpretation of this orientation, we refer the reader to~\cite{Chari2000}. 

Bergeron and Pilaud~\cite{BergeronPilaud2026} defined the \emph{hypergraphic poset} $P_\HH$ to be the \emph{transitive closure} of the $1$-skeleton of $\Delta_\HH$ oriented in the direction $\mathbf{\omega} := (n, n-1, \dots, 2, 1) - (1, 2, \dots, n-1, n) = (n-1, n-3, \dots, 3-n, 1-n)$. In particular, the Hasse diagram of $P_\HH$ is isomorphic to the \emph{transitive reduction} of this oriented $1$-skeleton. This definition recovers several well-known lattices. For example, if $\HH = {[n] \choose 2}$ (so that $\Delta_\HH$ is the permutahedron), then $P_\HH$ is the weak order on permutations; if $\HH = \{[i, j] \ | \ 1 \le i < j \le n\}$ (so that $\Delta_\HH$ is the associahedron), then $P_\HH$ is the Tamari lattice on binary trees~\cite{Tamari1951}. This motivates the question of characterizing the hypergraphs $\HH$ for which $P_\HH$ is a lattice, a distributive lattice, a semidistributive lattice, etc. This question has been answered for specific families of hypergraphs, such as graphs~\cite{Pilaud2024}, \emph{interval hypergraphs}~\cite{BergeronPilaud2026}, \emph{intreeval} hypergraphs~\cite{AbramBastidasGelinasPilaudSack2025}, and those whose hypergraphic polytopes are the multiplihedra and constrainahedra~\cite{ChapotonPilaud2024}. It has also been partially answered for graph associahedra~\cite{BarnardMcConville2021}.

A \emph{cyclic interval hypergraph} on $[n]$ is a hypergraph where every hyperedge is of the form $[i,j]$ or $[j,n] \cup [1, i]$ for some $1 \le i < j \le n$. In this paper, we present a complete lattice characterization of cyclic interval hypergraphic posets, generalizing that of \emph{interval hypergraphic posets} in~\cite{BergeronPilaud2026}. 

\begin{thm} \label{thmcyclic}
    Let $\HH$ be a cyclic interval hypergraph on $[n]$. The hypergraphic poset $P_\HH$ is a lattice if and only if for all $1 \le x < y \le n$, the restricted hypergraph $\HH' := \{H \cap [x, y] \ | \ H \in \HH\} \setminus \{\emptyset\}$ satisfies the following conditions:
    \begin{enumerate}
        \item the set $\HH'_{\reg}$ of all regular hyperedges in $\HH'$ is closed under intersection;
        \item every hugging quadruple of $\HH'$ has a fix.
    \end{enumerate}
\end{thm}
To avoid making the Introduction too technical, we defer introducing several definitions and notation used in the statement of the theorem to \Cref{sec:prelim}. 

We first note that \cite{BergeronPilaud2026} characterizes interval hypergraphic lattices as precisely those posets arising from interval hypergraphs that are closed under intersection. Therefore, when restricted to the interval case, our result coincides with their characterization. In \cite{BergeronPilaud2026}, the lattice characterization is obtained by realizing interval hypergraphic lattices as \emph{quasi-lattice quotients} of the Tamari lattice. Subsequently, the intreeval hypergraphic lattices are realized as the quasi-lattice quotients of the complete intreeval hypergraphic poset \cite{AbramBastidasGelinasPilaudSack2025}. It is thus tempting to extend this approach to the cyclic interval hypergraphs. However, the natural strategy of realizing cyclic interval lattices as quasi-lattice quotients of the complete cyclic interval hypergraphic poset (recently shown to be a lattice in \cite{AdenbaumBarnardHlavacekEtAl2025}) fails; check, for example, the hypergraph $\{12, 13, 14, 23, 24, 34, 124, 341\}$. Instead, we prove ~\Cref{thmcyclic} through explicit descriptions of joins and meets, given in \Cref{sec:join-meet-description}. These descriptions can also be reduced to the interval hypergraphic case studied in ~\cite{BergeronPilaud2026}.

While the hyperedges of an interval hypergraph $\II \subseteq \{[i, j] \ | \ 1 \le i < j \le n\}$ are all intervals in the usual sense (which we will call \emph{regular} hyperedges), a cyclic interval hyperpgraph can additionally have \emph{cyclic} hyperedges of the form $[j,n] \cup [1, i]$. These cyclic hyperedges substantially increase the difficulty of the lattice characterization problem. As mentioned above, \cite{AdenbaumBarnardHlavacekEtAl2025} shows that the complete cyclic interval hypergraphic poset is a lattice. This represents one example of our main result, which characterizes all cyclic interval hypergraphic lattices using a simpler and more efficient approach. Our proof is independent of their result and does not rely on it.

To obtain a full characterization, we begin with the observation that if $P_\HH$ is a lattice, then $P_{\HH'}$ is a lattice as well. Then, the main idea is to treat the regular and cyclic parts of $\HH$ separately when finding obstructions to being a lattice. The ``regular part" of the characterization coincides with the characterization for interval hypergraphs. For the ``cyclic part" of the characterization, we study a special structure involving four hyperedges, two cyclic and two regular. 

Our proof uses an alternative definition of $P_\HH$ called the \emph{source sequence poset} of $\HH$~\cite{Gelinas2025}. We can equivalently view each element of $P_\HH$ as an acyclic orientation of $\HH$ instead of a vertex of $\Delta_\HH$ \cite{BenedettiBergeronMachacek2019}. This perspective helps us explicitly construct a join for every pair of elements in $P_\HH$ for the ``if" direction of \Cref{thmcyclic}.

Throughout the paper, we assume $\{i\} \in \HH$ for all $i \in [n]$ for convenience. This is because every singleton set of $\HH$ contributes only a translation by a standard basis vector in the Minkowski sum and therefore does not affect the combinatorial structure of $\Delta_\HH$ and $P_\HH$. 

\section{Hypergraphic posets}\label{sec:prelim}

In this section, we introduce the definitions and notation that will be used throughout the paper. For background on \defn{partially ordered sets} (posets), we refer the reader to \cite[Ch. 3]{stanley2012enumerative}.

A \defn{hypergraph} $\HH$ on $[x, y]:= \{x, x+1, \dots, y-1, y\}$ is a set of subsets of $[x, y]$. The elements of $\HH$ are called \defn{hyperedges}. As previously mentioned, we assume all singletons $\{i\} \in \HH$ for $i \in [x, y]$. 

\begin{definition}
A \defn{cyclic interval hypergraph} $\HH$ on $[x,y]$ is a hypergraph where each $H \in \HH$ is of one of the two forms:
\begin{itemize}
    \item a \defn{regular hyperedge} $H = [i, j]$ such that $x \le i < j \le y$, and $H \ne [x, y]$. Let $\Hr$ denote the set of regular hyperedges in $\HH$.
    \item a \defn{cyclic hyperedge} $H = [j,y]\cup[x,i]$, where $x \le i < j \le y$. Let $\Hc$ denote the set of cyclic hyperedges in $\HH$. Moreover, we denote $H^- :=[x,i]$ and $H^+ := [j,y]$. In the case where $H = [x, y]$, then let $H^-= [x, y-1]$ and $H^+ = \{y\}$.
\end{itemize}
\end{definition}

Let $\HH$ be a hypergraph on $[x, y]$. We now introduce a combinatorial model for the hypergraphic poset $P_\HH$ defined in \Cref{sec:intro}. This is called the \defn{source sequence poset} of $\HH$. 

\begin{definition}\label{defn:orientation}
    An \defn{orientation} of $\HH$ is a map $A$ from $\HH$ to $[x,y]$ such that $A(H) \in H$ for all $H \in \HH$. We say that $A(H)$ is the \defn{source} of $H$ for each hyperedge $H$. The orientation $A$ is \defn{cyclic} if there is a sequence $S=(H_1, \dots , H_k)$ with $k \geq 2$ such that $A(H_{i+1}) \in H_i \setminus \{A(H_i)\}$ for $i \in [k - 1]$ and $A(H_1) \in H_k \setminus \{A(H_k)\}$. We say the orientation is \defn{acyclic} if no such sequence exists. 
\end{definition}

\begin{remark*}
    We mentioned before that the singleton elements of $\HH$ do not affect the combinatorial structure of $P_\HH$. Indeed, given any orientation $A$ and any singleton set $\{i\} \in \HH$, the source is always $A(\{i\}) = i$. Since $\{i\} \setminus A(\{i\}) = \emptyset$, the singleton sets does not affect the acyclicity of $A$. We will therefore omit any discussions about them in our proofs. 
\end{remark*}

\begin{definition}
Given a fixed order $(H_1, H_2,\dots, H_k)$ on the non-singleton elements of $\HH$, we describe an orientation $A$ with the tuple $S_A:=(A(H_1), A(H_2), \dots, A(H_k))$. This is called the \defn{source sequence} of $A$. 
\end{definition}

\begin{example}
    \Cref{orientation} illustrates the cyclic interval hypergraph $\HH=\{H_1=1256,H_2=123,H_3=23456\}$ and the graph of the orientation $A$ given by the following sources for each of its hyperedges:
    
    $$\begin{matrix}
        A(H_1)=6 && A(H_2)=2 && A(H_3)=4.
    \end{matrix}$$
\end{example}

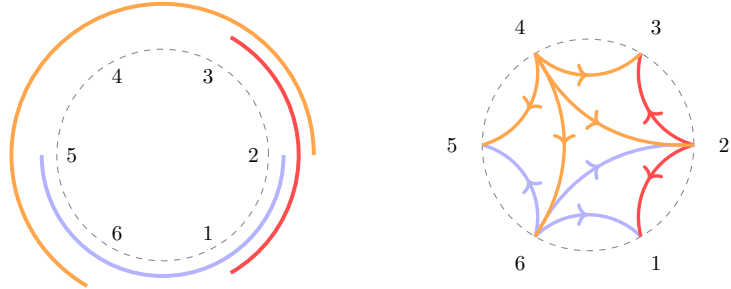
\begin{figure}[H]
 \scalebox{0.8}{\begin{tikzpicture}[decoration={markings, mark=at position 0.5 with {\arrow{>}}}]
     
 \node at (-3.5,0) {\begin{tikzpicture}
 
\def\R{1.75}
\def\RR{1.5}

\draw[dashed,gray] (0,0) circle (\R);

\draw[line width=1.85pt, draw=red!70] (0,0) ++(300:2.25) arc[start angle=300, end angle=420, radius=2.25];

\draw[line width=1.85pt, draw=orange!70] (0,0) ++(0:2.5) arc[start angle=0, end angle=240, radius=2.5];

\draw[line width=1.85pt, draw=blue!30] (0,0) ++(180:2) arc[start angle=180, end angle=360, radius=2];



\node at ({\RR*cos(300)},{\RR*sin(300)}) {$1$};
\node at ({\RR*cos(360)},{\RR*sin(360)}) {$2$};
\node at ({\RR*cos(60)},{\RR*sin(60)}) {$3$};
\node at ({\RR*cos(120)},{\RR*sin(120)}) {$4$};
\node at ({\RR*cos(180)},{\RR*sin(180)}) {$5$};
\node at ({\RR*cos(240)},{\RR*sin(240)}) {$6$};

\coordinate (1) at ({\R*cos(300)},{\R*sin(300)});
\coordinate (2) at ({\R*cos(360)},{\R*sin(360)});
\coordinate (3) at ({\R*cos(60)},{\R*sin(60)});
\coordinate (4) at ({\R*cos(120)},{\R*sin(120)});
\coordinate (5) at ({\R*cos(180)},{\R*sin(180)});
\coordinate (6) at ({\R*cos(240)},{\R*sin(240)});

\end{tikzpicture}};


 \node at (3.5,0) {\begin{tikzpicture}
 
\def\R{1.75}
\def\RR{2.25}

\draw[dashed,gray] (0,0) circle (\R);






\draw[ultra thick,blue!30,postaction={decorate}] (6) to[bend left=45] (1);
\draw[ultra thick,blue!30,postaction={decorate}] (6) to[bend left=33] (2);
\draw[ultra thick,blue!30,postaction={decorate}] (6) to[bend right=45] (5);

\draw[ultra thick,red!70,postaction={decorate}] (2) to[bend right=45] (1);
\draw[ultra thick,red!70,postaction={decorate}] (2) to[bend left=45] (3);

\draw[ultra thick,orange!70,postaction={decorate}] (4) to[bend left=33] (6);
\draw[ultra thick,orange!70,postaction={decorate}] (4) to[bend left=45] (5);
\draw[ultra thick,orange!70,postaction={decorate}] (4) to[bend right=45] (3);
\draw[ultra thick,orange!70,postaction={decorate}] (4) to[bend right=33] (2);




\node at ({\RR*cos(300)},{\RR*sin(300)}) {$1$};
\node at ({\RR*cos(360)},{\RR*sin(360)}) {$2$};
\node at ({\RR*cos(60)},{\RR*sin(60)}) {$3$};
\node at ({\RR*cos(120)},{\RR*sin(120)}) {$4$};
\node at ({\RR*cos(180)},{\RR*sin(180)}) {$5$};
\node at ({\RR*cos(240)},{\RR*sin(240)}) {$6$};

\end{tikzpicture}};
\end{tikzpicture}}
\caption{The cyclic interval hypergraph $\HH=\{\color{blue!30}{1256}\color{black}{,}\, \color{red!70}{123}\color{black}{,}\,\color{orange!70}{23456}\color{black}{\}}$ on the left and the orientation $A$ given by the source sequence $S_A=(6,2,4)$ on the right. Note that this orientation is acyclic.}
\label{orientation}\end{figure}

Throughout this paper, we use the following characterization of the hypergraphic poset $P_\HH$.
\begin{theorem}[\cite{Gelinas2025}] \label{sourcecharacterization}
    For any acyclic orientations $A, B \in P_\HH$, we have $A \le B$ if and only if $A(H) \le B(H)$ for all $H \in \HH$. 
\end{theorem}

\section{Key definition for the characterization}

In this section, we introduce a key structure in cyclic interval hypergraphs that will be used in our characterization. Let $\HH$ be a cyclic interval hypergraph on $[x, y]$.
\begin{definition}\label{def:hug}
    A collection of four distinct hyperedges $I, \Ic, J, \Jc$ is called a \defn{hugging quadruple} if:
    \begin{itemize}
        \item $I,J\in \Hr$ and $\Ic, \Jc\in \Hc$
        \item $\{x,x+1\}\subseteq I, \Ic$
        \item $\{y,y-1\}\subseteq J, \Jc$
    \end{itemize}
\end{definition}

\begin{definition} \label{def:hugquad}
    Let $\mathcal{Q}=\{I,\Ic,J,\Jc\}$ be a hugging quadruple. Then a \defn{fix} of $\mathcal{Q}$ is a hyperedge $H\in \HH$ such that $\{x+1,y-1\}\subseteq H \subseteq \bigcup_{K\in \mathcal{Q}} K$.
\end{definition} 

\begin{remark}
    Note that a fix of a hugging quadruple does not need to be distinct from the four hyperedges in the quadruple. 
\end{remark}

\begin{example} \label{ex:hugquad}
    \Cref{minimalhug} shows an example of two hypergraphs on $[1,4]$. The first one  is the hypergraph $\HH=\{12,124,34,134\}$ consisting of the following hugging quadruple:

    $$\begin{matrix}
        I=12, && \Ic=124, && J=34, && \Jc=134.
    \end{matrix}$$

    The second one is the hypergraph $\HH=\{12,124,23,34,134\}$ consisting of a hugging quadruple and its fix as follows:

    $$\begin{matrix}
        I=12, && \Ic=124, && J=34, && \Jc=134, && \textit{fix}=23.
    \end{matrix}$$

    \begin{figure}[H]
 \scalebox{0.95}{\begin{tikzpicture}
 
 \node at (-3.5,0) {\begin{tikzpicture}[scale=0.4, every node/.style={font=\small}]
\def\R{3}
\def\RR{2.5}
\node (v1) at ({\RR*cos(0)},{\RR*sin(0)}) {1};
\node (v2) at ({\RR*cos(90)},{\RR*sin(90)}) {2};
\node (v3) at ({\RR*cos(180)},{\RR*sin(180)}) {3};
\node (v4) at ({\RR*cos(270)},{\RR*sin(270)}) {4};
\draw[dashed,gray] (0,0) circle (\R);
\draw[line width=1.75pt, draw=red!70] (0,0) ++(0:3.50) arc[start angle=0, end angle=90, radius=3.50];
\draw[line width=1.75pt, draw=blue!30] (0,0) ++(180:3.50) arc[start angle=180, end angle=270, radius=3.50];
\draw[line width=1.75pt, draw=orange!70] (0,0) ++(270:4.50) arc[start angle=270, end angle=450, radius=4.50];
\draw[line width=1.75pt, draw=violet!50] (0,0) ++(180:5.00) arc[start angle=180, end angle=360, radius=5.00];

\end{tikzpicture}};

\node at (3.5,0) {\begin{tikzpicture}[scale=0.4, every node/.style={font=\small}]
\def\R{3}
\def\RR{2.5}
\node (v1) at ({\RR*cos(0)},{\RR*sin(0)}) {1};
\node (v2) at ({\RR*cos(90)},{\RR*sin(90)}) {2};
\node (v3) at ({\RR*cos(180)},{\RR*sin(180)}) {3};
\node (v4) at ({\RR*cos(270)},{\RR*sin(270)}) {4};
\draw[dashed,gray] (0,0) circle (\R);
\draw[line width=1.75pt, draw=red!70] (0,0) ++(0:3.50) arc[start angle=0, end angle=90, radius=3.50];
\draw[line width=1.75pt, draw=blue!30] (0,0) ++(180:3.50) arc[start angle=180, end angle=270, radius=3.50];
\draw[line width=1.75pt, draw=orange!70] (0,0) ++(270:4.50) arc[start angle=270, end angle=450, radius=4.50];
\draw[line width=1.75pt, draw=violet!50] (0,0) ++(180:5.00) arc[start angle=180, end angle=360, radius=5.00];
\draw[line width=1.75pt, draw=yellow] (0,0) ++(450:4.0) arc[start angle=450, end angle=540, radius=4.0];

\end{tikzpicture}};

\end{tikzpicture}}
\caption{The hypergraph $\HH=\{\textcolor{red!70}{12}, \textcolor{blue!30}{34}, \textcolor{orange!70}{124}, \textcolor{violet!70}{134}\}$ on the left consist of a hugging quadruple without a fix. The hypergraph $\HH=\{\textcolor{red!70}{12},\textcolor{yellow}{23}, \textcolor{blue!30}{34}, \textcolor{orange!70}{124}, \textcolor{violet!70}{134}\}$ on the right consist of a hugging quadruple and a fix. }
\label{minimalhug}\end{figure}
\end{example}

\begin{definition} \label{def:intersectionclosed}
    A hypergraph $\HH$ is \defn{closed under intersection} if for any hyperedges $H,H'\in \HH$ such that $H \cap H'\neq \emptyset$, their intersection $H \cap H'$ is a hyperedge in $\HH$.
\end{definition}

\section{Forward direction of \Cref{thmcyclic}}\label{nec}

To prove the forward direction of \Cref{thmcyclic}, we introduce a surjection $\mathcal{O}$, described in~\cite[Lem. $2.9$]{BenedettiBergeronMachacek2019}, from the set of permutations on $[n]$ to the acyclic orientations of a hypergraph on $[n]$, and a useful result which will form the crux of the proof.

\begin{definition} \label{surj}
    For a permutation $\pi$ of $[n]$, the \defn{orientation $\mathcal{O}_\pi$} of $\HH$ is defined for all $H \in \HH$ by
    \vspace{-0.05cm}
    $$\mathcal{O}_\pi(H):=\pi(\min\{j\ |\ \pi(j)\in H\}).$$
\end{definition}

\begin{proposition}[{\citetext{\citealp[Lem. 2.9]{BenedettiBergeronMachacek2019}; \citealp[Prop. 2.14]{BergeronPilaud2026}}}] \label{adjacent}
    For any hypergraph $\HH$ on $[n]$,
    \begin{itemize}
        \item the map $\mathcal{O}$ is a surjection from the permutations of $[n]$ to the acyclic orientations of $\HH$,
        \item the $1$-skeleton of $\Delta_\HH$ is isomorphic to the graph obtained by contracting the fibers of $\mathcal{O}$ in the graph of the permutahedron.
    \end{itemize} 
\end{proposition} 

Define $\HH|_{D} := \{ H \cap D \ | \ H \in \HH \} \setminus \{\emptyset\}$ to be the \defn{restriction} of $\HH$ to $D$. We start by proving a useful structural result. We caution the reader that an \defn{interval} in a poset $P$, as mentioned in the statement below, is a subset $I \subset P$ with the property that if $a \le b \le c$ and $a, c \in I$, then $b \in I$. This notion of interval should not be confused with intervals of the form $[x, y] = \{x, x+1, \dots, y-1, y\}$, which coincide with the above definition only when the poset is a subset of the natural numbers equipped with their usual total order.

\begin{proposition}\label{intervalrestriction}
    Let $\HH$ a hypergraph on $[n]$ and $D = [x,y]$ where $1\leq x < y \leq n$. Then $P_{\HH|_D}$ is isomorphic to an interval in $P_\HH$. In particular, if $P_{\HH|_D}$ is not a lattice, then $P_\HH$ is not a lattice.
\end{proposition}
\begin{proof}
By \Cref{adjacent}, there exists a permutation $\pi_A$ on $D$ such that $\mathcal{O}_{\pi_A}=A$ for every acyclic orientation $A \in P_{\HH|_D}$. Let $X$ be a permutation on $[n]\setminus D$. Then $\mathcal{O}_{\pi_A X} \in P_\HH$. Moreover, for every hyperedge $H \in \HH$, the source $\mathcal{O}_{\pi_A X} (H)$ is in $D$ if $H \cap D \ne \emptyset$ and in $[n]\setminus D$ otherwise. 

We will show that $\mathcal{D}:= \{\mathcal{O}_{\pi_A X}\ | \ A \in P_{\HH|_D}\}$ is an interval in $P_\HH$. Consider $A, B \in P_{\HH|_D}$ such that $\mathcal{O}_{\pi_A X} \le \mathcal{O}_{\pi_B X}$ and $\gamma \in P_\HH$ such that $\mathcal{O}_{\pi_A X} \le \gamma \le \mathcal{O}_{\pi_B X}$. Consider a hyperedge $H \in \HH$. If $H \cap D \ne \emptyset$, then $\mathcal{O}_{\pi_A X} (H) \le \gamma (H) \le \mathcal{O}_{\pi_B X} (H)$. Since $D = [x, y]$, it must be that $\gamma(H) \in D$. If $H \cap D = \emptyset$, then $\mathcal{O}_{\pi_A X} (H) = \gamma (H) = \mathcal{O}_{\pi_B X} (H)$.  Thus $\gamma \in \mathcal{D}$, as desired.  
\end{proof}

The following example illustrates a subtlety that can occur when a hypergraph is restricted to a given interval $D=[x,y]$.
\begin{example}
    Consider two hypergraphs $\HH=\{125,345,1245,1345\}$ and $\HH'=\{234,1235,1345\}$ on $[5]$. 
    
    From \Cref{def:hug,def:hugquad}, one can check that there is no hugging quadruple in $\HH$. However, if we take $D:=[1,4]$, then the restriction $\HH|_D =\{12,34,124,134\}$, which is the first hypergraph in \Cref{ex:hugquad}, now admits a hugging quadruple.

    By \Cref{def:intersectionclosed}, $\HH'_{\text{reg}} = \{234\}$ is closed under intersection. However, if we take $D:=[1,4]$ again, then $(\HH'|_D)_{\text{reg}}$ is not closed under intersection. Indeed, $H_1:=234$ and $H_2:=123$ are both in $(\HH'|_D)_{\text{reg}}$ but $H_1 \cap H_2 =23$ is not a hyperedge in $(\HH'|_D)_{\text{reg}}$.
\end{example}

We are now ready to prove the forward direction of \Cref{thmcyclic}.

\begin{proposition}\label{prop:necessary}
     Let $\HH$ be a cyclic interval hypergraph on $[n]$. If the hypergraphic poset $P_\HH$ is a lattice, then ${(\HH|_D)}_{\reg}$ is closed under intersection and there exist a fix for every hugging quadruple of $\HH|_D$ for any interval $D = [x, y]$ where $1 \le x<y \le n$.
\end{proposition}

\begin{proof}
    Suppose $P_\HH$ is a lattice. The condition of ${(\HH|_D)}_{\reg}$ being closed under intersection follows from~\Cref{intervalrestriction} and~\cite[Prop. $4.6$]{BergeronPilaud2026}. The rest of the proof expands on the same techniques as in ~\cite[Prop. $4.6$]{BergeronPilaud2026}. 

    For the sake of a contradiction, suppose for an arbitrary $D=[x,y]$ such that $1 \leq x < y \leq n$, there is a hugging quadruple $\{I,\Ic,J,\Jc\}$ without a fix. This implies that $n\geq 4$. \Cref{intervalrestriction} allows us to assume without loss of generality that $D=[n]$. Consider the following permutations:

    $$\begin{matrix}
    \pi_A=Y(n-1)12nX, && \pi_B=Y1n2(n-1)X, \\ \pi_C=Yn1(n-1)2X, && \pi_D=Y2n(n-1)1X,
    \end{matrix}$$
    where $Y$ is any permutation on $[n]\setminus (I \cup \Ic \cup J \cup \Jc)$ and where $X$ is any permutation on $[n]\setminus (Y \cup \{1,2,n-1,n\})$. Their corresponding acyclic orientations are as follows:

    $$\begin{matrix}
    A:=\mathcal{O}_{\pi_A}, && B:=\mathcal{O}_{\pi_B}, && C:=\mathcal{O}_{\pi_C}, && D:=\mathcal{O}_{\pi_D}.
    \end{matrix}$$

    We check the first four positions of the source sequences of these orientations displayed in the order $I,\Ic,J,\Jc$. Since $I,\Ic,J,\Jc$ are not fixes, we have that $n\notin I$, $n-1\notin I \cup \Ic$,  $1\notin J$ and $2\notin J \cup \Jc$ and thus the following:

    $$\begin{matrix}
    S_A=(1,1,n-1,n-1,\dots) , && S_B=(1,1,n,1,\dots) ,\\S_C=(1,n,n,n,\dots) , && S_D=(2,2,n,n,\dots) ,
    \end{matrix}$$

    By \Cref{sourcecharacterization}, $A\not<B$ and $C\not<D$. By~\cite[Cor. 2.15]{BergeronPilaud2026}, $B<C$ according to the weak order on $4$ elements. Since the only hyperedges $H$ in $\HH$ that do not contain $\{n\}\cup Y$ but contain $1$ and $n-1$ would be a fix, we have that $\mathcal{O}_{\pi_C}=\mathcal{O}_{Yn(n-1)12X}$ and thus by \cite[Cor. 2.15]{BergeronPilaud2026} the weak order implies that $A<C$. Since the only hyperedges $H$ in $\HH$ that do not contain $Y$ but contain $2$ and $n-1$ would be a fix, we have that $\mathcal{O}_{Y(n-1)21nX}=\mathcal{O}_{Y2(n-1)1nX}$. By~\cite[Cor. 2.15]{BergeronPilaud2026}, the weak order implies that $A<\mathcal{O}_{Y(n-1)21nX}=\mathcal{O}_{Y2(n-1)1nX}<D$. Finally, the only hyperedges $H$ in $\HH$ that do not contain $\{1\}\cup Y$ but contain $2$ and $n$ would be a fix, so $\mathcal{O}_{\pi_B}=\mathcal{O}_{Y12n(n-1)X}$ and the weak order implies that $B<D$. See \Cref{permu}.

    \input{Drawings/permu}

    Since $P_\HH$ is a lattice, there exists an orientation $M$ such that $A, B \leq M$ and $M\leq C, D$. By \Cref{sourcecharacterization}, $S_M$ is one of the following source sequences:
    \[
    (1,1,n,n-1,\dots), \ (1,1,n,n,\dots), \ (1,2,n,n-1,\dots), \ (1,2,n,n,\dots).
    \]
    However, these four orientations are all cyclic, contradicting the existence of a join, which concludes the proof.
\end{proof}

\section{Backward direction of \Cref{thmcyclic}}\label{sec:backward}
The goal of this section is to prove the ``if" direction of \Cref{thmcyclic}. Given a cyclic interval hypergraph $\HH$ satisfying the hypothesis, we would like to show that there exists a join for every pair of acyclic orientations $A$ and $B$ in its source sequence poset $P_\HH$. In \Cref{subsec:pseudojoin}, we propose a \emph{pseudo-join} $X^{AB}$ for $A$ and $B$. In \Cref{subsec:acyclicity}, we prove that this pseudo-join is acyclic and therefore is an element of $P_\HH$. We first show in \Cref{prop:minimalcycle} that any minimal cycle in $X^{AB}$ would have length $2$. Then we analyze the case of $2$-cycles in \Cref{lem:twocycle-reg-cyc}, \Cref{prop:twocycle-reg-cyc}, and \Cref{prop:twocycle}, eventually arriving at a contradiction by finding a restriction of $\HH$ that contains a hugging quadruple without a fix. This section contains several technical lemmas to assist the proofs. Finally, we show that the pseudo-join is indeed the join in \Cref{join} in \Cref{subsec:join}.

Throughout this section, let $\HH$ be a cyclic interval hypergraph on $[n]$ and let $P_\HH$ be its source sequence poset.

\subsection{Pseudo-join}\label{subsec:pseudojoin}

For every hyperedge $H \in \HH$ and two distinct acyclic orientations $A$ and $B$ in $P_\HH$, let \defn{$H_{AB}$}$:=\max\{A(H), B(H)\}$. 

\begin{definition}
Let $H \in \HH$ and $\ell \in H$. Let $A$ and $B$ be two distinct acyclic orientations of $\HH$. We say that a squence $(H_i, h_i)_{i = 1}^k$ where $H_i \in \HH$ and $h_i \in H_i$ is an \defn{$(H, \ell)$-sequence with respect to $A$ and $B$} if the following conditions are satisfied:
\begin{itemize}
    \item $H_1 = H$
    \item $h_1 = \ell$ and $h_i = (H_i)_{AB}$ if $i \ne 1$
    \item $h_i \in H_{i+1}$ and $h_i < h_{i+1}$ for every $i \in [k-1]$
    \item $h_k \in H$
\end{itemize}
Moreover, let 
\[X^{AB}(H, \ell) = \max \{h_k\ | \ (H_i, h_i)_{i = 1}^k \text{ is an $(H, \ell)$-sequence with respect to $A$ and $B$}\}.
\]
Define 
\[
X^{AB}(H) = \min_{\ell \in H, \ \ell \ge H_{AB}} \{X(H, \ell)\}.
\]
We call $X^{AB}$ the \defn{pseudo-join} of $A$ and $B$. We omit the notation about $A$ and $B$ when they are clear from the context.
\end{definition}

\begin{example}\label{ex:pseudojoin}
\Cref{pseudojoin} illustrates the procedure to obtain the pseudo-join of two given orientations $A,B$ of a hypergraph $\HH$. Consider the cyclic interval hypergraph 

$$\HH=\{H_1=1236, H_2=234, H_3=1256, H_4=34, H_5= 56\}.$$

Consider the orientations $A$ and $B$ respectively given by the source sequences $S_A=(2,2,2,4,6)$ and $S_B=(1,3,5,3,5)$. We compute the $(H,\ell)$-sequences for each $H\in\HH$ and $\ell\geq H_{AB}$ as follows.
$$\begin{matrix*}[l]
    (H_1,2) : (H_1),(H_1,H_2),(H_1,H_3,H_5) && (H_1,3) : (H_1) &&  (H_1,6) : (H_1),(H_1,H_3,H_5)\\
    (H_2,3) : (H_2),(H_2,H_4) && (H_2,4) : (H_2) && && \\
    (H_3,5) : (H_3),(H_3,H_5) && (H_3,6) : (H_3)  &&  && \\
    (H_4,4) : (H_4)&&  &&  &&  \\
    (H_5,6) : (H_5). && && &&
\end{matrix*}$$

Next we compute $X^{AB}(H,\ell)$ for each $H\in \HH$ and each $\ell\geq H_{AB}$.
$$\begin{matrix}
    X^{AB}(H_1,2)=6 && X^{AB}(H_1,3)=3 && X^{AB}(H_1,6)=6 \\
    X^{AB}(H_2,3)=4 && X^{AB}(H_2,4)=4 \\
    X^{AB}(H_3,5)=6 && X^{AB}(H_3,6)=6 \\
    X^{AB}(H_4,4)=4  \\
    X^{AB}(H_5,6)=6.
\end{matrix}$$

We can now obtain the source of the pseudo-join $X^{AB}(H)$ for each $H\in \HH$.
$$\begin{matrix}
    X^{AB}(H_1) =\min\{3,6\}=3 & X^{AB}(H_2) =4 & X^{AB}(H_3) = 6 & X^{AB}(H_4) = 4 & X^{AB}(H_5) = 6.
\end{matrix}$$
Therefore, the pseudo-join $X$ is given by the source sequence $S_X=(3,4,6,4,6).$
\end{example}

\begin{figure}[H]
 \scalebox{0.8}{\begin{tikzpicture}[decoration={markings, mark=at position 0.5 with {\arrow{>}}}]
     
 \node at (-3.5,0) {\begin{tikzpicture}
 
\def\R{1.75}
\def\RR{1.5}

\draw[dashed,gray] (0,0) circle (\R);

\draw[line width=1.75pt, draw=Black!80] (0,0) ++(360:2.25) arc[start angle=360, end angle=480, radius=2.25];

\draw[line width=1.75pt, draw=Black!40] (0,0) ++(420:2.5) arc[start angle=420, end angle=480, radius=2.5];

\draw[line width=1.75pt, draw=Black] (0,0) ++(240:2) arc[start angle=240, end angle=420, radius=2];

\draw[line width=1.75pt, draw=Black!60] (0,0) ++(180:2.5) arc[start angle=180, end angle=360, radius=2.5];

\draw[line width=1.75pt, draw=Black!20] (0,0) ++(180:2.25) arc[start angle=180, end angle=240, radius=2.25];


\node at ({\RR*cos(300)},{\RR*sin(300)}) {$1$};
\node at ({\RR*cos(360)},{\RR*sin(360)}) {$2$};
\node at ({\RR*cos(60)},{\RR*sin(60)}) {$3$};
\node at ({\RR*cos(120)},{\RR*sin(120)}) {$4$};
\node at ({\RR*cos(180)},{\RR*sin(180)}) {$5$};
\node at ({\RR*cos(240)},{\RR*sin(240)}) {$6$};

\node at ({2*cos(300)},{2*sin(300)}) {\Large{$\textcolor{blue!30}{\bullet}$}};
\node at ({2*cos(360)},{2*sin(360)}) {\Large{$\textcolor{red!70}{\bullet}$}};
\node at ({2.25*cos(360)},{2.25*sin(360)}) {\Large{$\textcolor{red!70}{\bullet}$}};
\node at ({2.5*cos(360)},{2.5*sin(360)}) {\Large{$\textcolor{red!70}{\bullet}$}};
\node at ({2.5*cos(420)},{2.5*sin(420)}) {\Large{$\textcolor{blue!30}{\bullet}$}};
\node at ({2.25*cos(420)},{2.25*sin(420)}) {\Large{$\textcolor{blue!30}{\bullet}$}};
\node at ({2.5*cos(480)},{2.5*sin(480)}) {\Large{$\textcolor{red!70}{\bullet}$}};
\node at ({2.25*cos(180)},{2.25*sin(180)}) {\Large{$\textcolor{blue!30}{\bullet}$}};
\node at ({2.5*cos(180)},{2.5*sin(180)}) {\Large{$\textcolor{blue!30}{\bullet}$}};
\node at ({2.25*cos(240)},{2.25*sin(240)}) {\Large{$\textcolor{red!70}{\bullet}$}};

\end{tikzpicture}};


 \node at (3.5,0) {\begin{tikzpicture}
 
\def\R{1.75}
\def\RR{1.5}

\draw[dashed,gray] (0,0) circle (\R);

\draw[line width=1.75pt, draw=Black!80] (0,0) ++(360:2.25) arc[start angle=360, end angle=480, radius=2.25];

\draw[line width=1.75pt, draw=Black!40] (0,0) ++(420:2.5) arc[start angle=420, end angle=480, radius=2.5];

\draw[line width=1.75pt, draw=Black] (0,0) ++(240:2) arc[start angle=240, end angle=420, radius=2];

\draw[line width=1.75pt, draw=Black!60] (0,0) ++(180:2.5) arc[start angle=180, end angle=360, radius=2.5];

\draw[line width=1.75pt, draw=Black!20] (0,0) ++(180:2.25) arc[start angle=180, end angle=240, radius=2.25];


\node at ({\RR*cos(300)},{\RR*sin(300)}) {$1$};
\node at ({\RR*cos(360)},{\RR*sin(360)}) {$2$};
\node at ({\RR*cos(60)},{\RR*sin(60)}) {$3$};
\node at ({\RR*cos(120)},{\RR*sin(120)}) {$4$};
\node at ({\RR*cos(180)},{\RR*sin(180)}) {$5$};
\node at ({\RR*cos(240)},{\RR*sin(240)}) {$6$};

\node at ({2.5*cos(120)},{2.5*sin(120)}) {\Large{$\textcolor{orange!70}{\bullet}$}};
\node at ({2.25*cos(120)},{2.25*sin(120)}) {\Large{$\textcolor{orange!70}{\bullet}$}};
\node at ({2*cos(60)},{2*sin(60)}) {\Large{$\textcolor{orange!70}{\bullet}$}};
\node at ({2.25*cos(240)},{2.25*sin(240)}) {\Large{$\textcolor{orange!70}{\bullet}$}};
\node at ({2.5*cos(240)},{2.5*sin(240)}) {\Large{{$\textcolor{orange!70}{\bullet}$}}};

\end{tikzpicture}};
\end{tikzpicture}}
\caption{The cyclic interval hypergraph $\HH=\{\textcolor{Black}{1236}, \textcolor{Black!80}{234},\textcolor{Black!60}{1256},\textcolor{Black!40}{34},\textcolor{Black!20}{56}\}$. On the left we can see the juxtaposition of the orientations $A$ with sources represented with $\textcolor{blue!30}{\bullet}$ and $B$ with sources represented with $\textcolor{red!70}{\bullet}$ respectively given by the source sequence $S_A=(1,3,5,3,5)$ and $S_B=(2,2,2,4,6)$. On the right we can see the orientation given by the pseudo-join $X$ of $A$ and $B$ given by the source sequence $S_X=(3,4,6,4,6)$. The sources of $X$ are represented with $\textcolor{orange!70}{\bullet}$. Note that all orientations here are acyclic.}
\label{pseudojoin}\end{figure}

\subsection{Structural results}

Given two sequences $(a_i)_{i=1}^k$ and $(b_i)_{i=1}^r$, let $(a_i)_{i=1}^k \| (b_i)_{i=1}^r$ denote the sequence obtained by concatenating $(a_i)_{i=1}^k$ by $(b_i)_{i=1}^r$, i.e., $(a_1, \dots, a_k, b_1, \dots, b_r)$.

\begin{lemma}\label{lem:Hreg}
    Let $A$ and $B$ be two distinct acyclic orientations of $\HH$. If $H \in \Hr$, then: \begin{enumerate}
        \item For every $h_j,h_{j+1}$ that appears in an $(H, \ell)$-sequence, $[h_j, h_{j+1}]\subseteq H_{j+1}$. 
        \item $X(H) = X(H, H_{AB})$.
    \end{enumerate}  
    
\end{lemma} 
\begin{proof} \mbox{} \begin{enumerate}
    \item

    Suppose for the sake of contradiction that $[h_j, h_{j+1}]\nsubseteq H_{j+1}$. Then $H_{j+1} \in \Hc$ and $h_{j+1}\in H_{j+1}^+$. Since $H_{AB} \le h_j < h_{j+1}$, this implies that $A(H), B(H) \in H_{j+1}^- \subset H_{j+1}$. Meanwhile, since $H \in \Hr$, we have that $[H_{AB},h_{j+1}]\subset H$. Hence, either $A$ or $B$ is cyclic, which is a contradiction.

    \item

    Suppose for the sake of contradiction that there exists an element $\ell\in H$ such that $H_{AB}<l\leq X(H,\ell)<X(H,H_{AB})$.  Let $(H_i, h_i)_{i = 1}^k$ be a sequence that realizes $X(H,H_{AB})$. By construction of $X$ and by $(1)$, there must be $h_j,h_{j+1}$ such that $X(H,\ell)\in [h_j,h_{j+1})$ which implies that the $(H, \ell)$-sequence realizing $X(H,\ell)$ can be concatenated with the sequence $(H_i, h_i)_{i = j+1}^k$ which contradicts $X(H,\ell)<X(H,H_{AB})$.
\end{enumerate}
\end{proof}

The proof of \Cref{lem:Hreg} captures two core arguments used throughout \Cref{sec:backward}. The first part discusses whether the interval $[h_j, h_{j+1}]$ is contained in $H_{j+1}$, a hyperedge in an $(H, \ell)$-sequence. This would not be a question if $H_{j+1} \in \Hr$. However, since we are dealing with potentially cyclic hyperedges, we need to take more into consideration. By the argument for \Cref{lem:Hreg}(1), whenever $h_{j+1} \in H$, we must have $[h_j, h_{j+1}] \subseteq H_{j+1}$ by the acyclicity of $A$ and $B$, regardless of $H \in \Hr$ or $H \in \Hc$. Even without $h_{j+1} \in H$, the same argument can be used when our hypothesis forbids that a certain $x < h_j$ be contained in $H_{j+1}$. 

The second part utilizes the definition of the pseudo-join. Let $(H'_i, h'_i)_{i=1}^s$ be a sequence realizing $X(H', \ell') < X(H, \ell) \in H'$ ($H$ and $H'$ are not necessarily distinct), then $h'_j \notin H_t$ for any $j \in [s]$ and $t \in [k]$ such that $h'_j < h_t = (H_t)_{AB}$. Otherwise, $(H'_i, h'_i)_{i=1}^j\| (H_i, h_i)_{i=t}^k$ is an $(H', \ell')$-sequence.

\begin{lemma}{\label{lem:wrap2}} 
Let $A$ and $B$ be two distinct acyclic orientations of $\HH$ and let $H\in \Hc$. If $X(H)\neq X(H,H_{AB})$ then there exists $\Bar{H}\in \Hc$ such that:
    \begin{enumerate}
        \item $[1,\ell-1]\subset \Bar{H}$, where $\ell$ is the smallest index such that $X(H)=X(H,\ell)$;
        \item $X(\Bar{H})=X(H,H_{AB})\in H^+$. 
    \end{enumerate}
\end{lemma} 

\begin{proof}
    Consider the sequence $(H_i,h_i)_{i=1}^k$ that realizes $X(H,\ell-1)$. Let $\Bar{H} = H_2$, so $\ell - 1 \in \bar{H}$. Since $X(H)\notin \Bar{H}$, otherwise we can extend the sequence realizing $X(H, \ell)$ to $(H_i,h_i)_{i=1}^k$, we have that $[1,\ell-1]\subset \Bar{H}$. The acyclicity of $A$ and $B$ implies $\Bar{H}_{AB}\notin H$, so $\Bar{H}_{AB}\leq \min(H^+)$. Therefore, $H^+\subset \Bar{H}^+$ and $X(\Bar{H})=X(H,\ell-1)\in H^+$ and the result follows by concatenation. 
\end{proof}

\begin{example}[\Cref{ex:pseudojoin} continued]
    \Cref{pseudojoin} illustrates the two parts of \Cref{lem:wrap2}. Recall that the example examines the cyclic interval hypergraph 

$$\HH=\{H_1=1236, H_2=234, H_3=1256, H_4=34, H_5= 56\}.$$

Consider the orientations $A$ and $B$ given by the source sequences $S_A=(2,2,2,4,6)$ and $S_B=(1,3,5,3,5)$, and the cyclic hyperedge $H_1\in \Hc$. Observe from \Cref{ex:pseudojoin} that $X(H_1)=X(H_1,3)$ and that $X(H_1)\neq X(H_1,(H_1)_{AB})=X(H_1,2)$. Following the notation from \Cref{lem:wrap2}, we have that $\ell = 3$ and $\Bar{H}=H_3$. Indeed, $[1,\ell-1=2]\subset H_3$ and $X(H_3)=6=X(H_1,2)$.
\end{example}

\subsection{Acyclicity}\label{subsec:acyclicity}

\begin{proposition}\label{prop:minimalcycle}
    Let $A,B$ be two acyclic orientation of $\HH$. If the pseudo-join $X^{AB}$ is cyclic then all minimal cycle $S=(H_1,\dots,H_s)$ are of length $2$.
\end{proposition}

\begin{proof}

    Suppose that the pseudo-join $X^{AB}$ is cyclic according to \Cref{defn:orientation} with a minimal cycle $S=(H_1,\dots,H_s)$ of length $s>2$. This implies that $X(H_{i-1})\notin X(H_i)$. Minimality implies that $X(H_{1})\notin H_{s-1}$ and $X(H_{i+1})\notin H_{i-1}$ for $1< i < s$. We can either assume that $X(H_i)<X(H_{i+1})$ or that $X(H_i)>X(H_{i+1})$ for $1\leq i < s$, where the cycles are respectively counterclockwise and clockwise. See \Cref{clock} for an illustration.

    \begin{figure}[H]
 \scalebox{0.95}{\begin{tikzpicture}
 
 \node at (-3.5,0) {\begin{tikzpicture}[scale=0.4, every node/.style={font=\small}]
\def\R{3}
\def\RR{2.5}
\def\RRR{3.5}
\node (v1) at ({\RR*cos(0)},{\RR*sin(0)}) {1};
\node (v2) at ({\RR*cos(45)},{\RR*sin(45)}) {2};
\node (v3) at ({\RR*cos(90)},{\RR*sin(90)}) {3};
\node  at ({\RR*cos(135)},{\RR*sin(135)}) {4};
\node  at ({\RR*cos(180)},{\RR*sin(180)}) {5};
\node  at ({\RR*cos(225)},{\RR*sin(225)}) {6};
\node  at ({\RR*cos(270)},{\RR*sin(270)}) {7};
\node  at ({\RR*cos(315)},{\RR*sin(315)}) {8};
\draw[dashed,gray] (0,0) circle (\R);
\draw[line width=1.75pt, color=blue!30] (0,0) ++(-15:4.50) arc[start angle=-15, end angle=105, radius=4.50];
\draw[line width=1.75pt,color=red!70] (0,0) ++(75:4.0) arc[start angle=75, end angle=195, radius=4.0];
\draw[line width=1.75pt,color=violet!70] (0,0) ++(165:4.50) arc[start angle=165, end angle=285, radius=4.50];
\draw[line width=1.75pt,color=orange!70] (0,0) ++(255:4.00) arc[start angle=255, end angle=375, radius=4.00];

\node at ({4*cos(270)},{4*sin(270)}) {\Large{$\bullet$}};
\node at ({4.5*cos(0)},{4.5*sin(0)}) {\Large{$\bullet$}};
\node at ({4*cos(90)},{4*sin(90)}) {\Large{$\bullet$}};
\node at ({4.5*cos(180)},{4.5*sin(180)}) {\Large{$\bullet$}};

\node at ({5.5*cos(45)},{5.5*sin(45)}) {$H_1$};
\node at ({5.5*cos(135)},{5.5*sin(135)}) {$H_2$};
\node at ({5.5*cos(225)},{5.5*sin(225)}) {$H_3$};
\node at ({5.5*cos(315)},{5.5*sin(315)}) {$H_4$};

\node at (0,0) {$S$};

\draw [->,line width=1pt] ++(140:10mm) arc (-220:40:10mm) --++(110:2mm);
\end{tikzpicture}};

\node at (3.5,0) {\begin{tikzpicture}[scale=0.4, every node/.style={font=\small}]
\def\R{3}
\def\RR{2.5}
\def\RRR{3.5}
\node (v1) at ({\RR*cos(0)},{\RR*sin(0)}) {1};
\node (v2) at ({\RR*cos(45)},{\RR*sin(45)}) {2};
\node (v3) at ({\RR*cos(90)},{\RR*sin(90)}) {3};
\node  at ({\RR*cos(135)},{\RR*sin(135)}) {4};
\node  at ({\RR*cos(180)},{\RR*sin(180)}) {5};
\node  at ({\RR*cos(225)},{\RR*sin(225)}) {6};
\node  at ({\RR*cos(270)},{\RR*sin(270)}) {7};
\node  at ({\RR*cos(315)},{\RR*sin(315)}) {8};
\draw[dashed,gray] (0,0) circle (\R);
\draw[line width=1.75pt,color=blue!30] (0,0) ++(-15:4.50) arc[start angle=-15, end angle=105, radius=4.50];
\draw[line width=1.75pt,color=red!70] (0,0) ++(75:4.0) arc[start angle=75, end angle=195, radius=4.0];
\draw[line width=1.75pt,color=violet!70] (0,0) ++(165:4.50) arc[start angle=165, end angle=285, radius=4.50];
\draw[line width=1.75pt,color=orange!70] (0,0) ++(255:4.00) arc[start angle=255, end angle=375, radius=4.00];

\node at ({4.5*cos(270)},{4.5*sin(270)}) {\Large{$\bullet$}};
\node at ({4*cos(0)},{4*sin(0)}) {\Large{$\bullet$}};
\node at ({4.5*cos(90)},{4.5*sin(90)}) {\Large{$\bullet$}};
\node at ({4*cos(180)},{4*sin(180)}) {\Large{$\bullet$}};

\node at ({5.5*cos(45)},{5.5*sin(45)}) {$H_3$};
\node at ({5.5*cos(135)},{5.5*sin(135)}) {$H_2$};
\node at ({5.5*cos(225)},{5.5*sin(225)}) {$H_1$};
\node at ({5.5*cos(315)},{5.5*sin(315)}) {$H_4$};

\node at (0,0) {$S$};
\draw [<-,line width=1pt] ++(138:10mm) --++(60:-1pt) arc (-220:40:10mm) ;
\end{tikzpicture}};

\end{tikzpicture}}
\caption{The hypergraph $\HH=\{\textcolor{blue!30}{123},\textcolor{red!70}{345},\textcolor{violet!70}{567},\textcolor{orange!70}{178}\}$. On the left, an counterclockwise minimal cycle $S=(178,123,345,567)$ given by the source sequence $(1,3,5,7)$. On the right, a clockwise minimal cycle $S=(567,345,123,178)$ given by the source sequence $(3,5,7,1)$. }
\label{clock}\end{figure}

    \emph{Case 1} (Counterclockwise). Suppose that for all $H_i\in S$ we have that $A(H_i)$ (resp. $B(H_i)$) and $X(H_i)$ are both in $H_i^+$ or $ H_i^-$. Since $X(H_{i+1})\in H_i$ and $X(H_1)\in H_s$, this implies that $A$ is cyclic (resp. $B$). Since $A(H_i)\leq X(H_i)$ (resp. $B(H_i)$), there exists a hyperedge $H_r\in S$ such that $A(H_r)\in H_r^-$ and $X(H_r)\in H_r^+$ and a hyperedge $H_k\in S$ such that $B(H_k)\in H_k^-$ and $X(H_k)\in H_k^+$.
    
    Suppose that $k\neq r$. Without loss of generality, assume that $k<r$. This implies that $X(H_k)<X(H_r)$. Since both $H_k$ and $H_r$ are cyclic, we can assume that $r=s$. By the minimality of $S$, we have that $H_k^- \subseteq H_r^-$ and $H_r^+ \subseteq H_k^+$. Therefore, $r = s$. Otherwise, $X(H_{r+1}) \notin H_k$ would force $X(H_{r+1})\in H_r^-$ a contradiction. Then $X(H) < X(H_r)$ for $H \in \{H_1, \dots, H_{k-1}\}$. The minimality of $S$ also implies that $X(H_r) \notin H$. Therefore, $X(H) \in H^-$ and thus, $A(H),B(H)\in H^-$ as well.

    By acyclicity of $B$ and the containments $H_k^- \subseteq H_r^-$ and $H_r^+ \subseteq H_k^+$, we have $B(H_r) \in H_r^-$ and $B(H_k) \le (H_r)_{AB}$. Consider the sequence $(\Bar{H}_i,\Bar{h}_i)_{i=1}^v$ that realizes $X(H_r,(H_r)_{AB})$.

    Since $X(H_r)\in H_k$, arguing by concatenation, no hyperedge $\Bar{H}_i$ contains $X(H_k)$. This implies the existence of a hyperedge $\Bar{H}_j$ such that $\Bar{H}_j\in \Hc$ and such that $\Bar{h}_j\in \Bar{H}_j^+$ and such that $A(H_r),B(H_r)\in \Bar{H}_j$. By acyclicity of $A$ and $B$, we have that $X(H_k)<\Bar{h}_j<\min(H_r^+)$ and that $\Bar{h}_j\neq B(\Bar{H}_j)$. Therefore, $\min(H_k^+) \leq A(H_k) \leq X(H_k)$. Since $X(H)\in H^-$ for all hyperedges $H\in \{H_1,\dots,H_{k-1}\}$, there exists a cycle in $A$ given by
        
    $$A(H_r) \xrightarrow{H_r} A(H_1) \xrightarrow{H_1} A(H_2) \xrightarrow{H_2} \cdots \xrightarrow{H_{k-1}} A(H_k) \xrightarrow{H_k} A(\Bar{H}_j) \xrightarrow{\Bar{H}_j} A(H_r).$$

    Suppose $k = r$. ThaPt is, $A(H_k), B(H_k) \in H_k^-$. If $k \ne s$, then $X(H_{k+1}) \in H_{k+1}^+$ and $A(H_{k+1}) \in H_{k+1}^-$ by the acyclicity of $A$, so we can return to the previous case. If $k = s$ and no other edge satisfies the same condition on $A$ or $B$, we repeat the argument above with $\Bar{H}_{j+1}$ such that $X(H_{k-1}) \in [\Bar{h}_j,\Bar{h}_{j+1})$. Assuming without loss of generality that $\Bar{h}_{j+1} = A(\Bar{H}_{j+1})$, we again find a cycle in $A$:
        \[
        A(H_k) \xrightarrow{H_k} A(H_1) \xrightarrow{H_1} A(H_2) \xrightarrow{H_2} \cdots \xrightarrow{H_{k-2}} A(H_{k-1}) \xrightarrow{H_{k-1}} A(\Bar{H}_{j+1}) \xrightarrow{\Bar{H}_{j+1}} A(H_k).
        \]

    \emph{Case 2} (Clockwise). Suppose that for the sake of contradiction that $X(H)=X(H,H_{AB})$ for all $H\in S$. This contradicts the definition of the pseudo-join since we would have that $X(H_1)=X(H_s)$. Hence, assume that $X(H_i)\neq X(H_i,(H_i)_{AB})$ for all $H_i\in S$. The minimality of the cycle $S$ implies that $i\in \{s-1,s\}$. Let $H_i$ be the first such hyperedge in the cycle. In the case where $i=s$, we can use the same above arguments to prove the the existence of another hyperedge $H_j\in S$ such that $X(H_j)\neq X(H_j,(H_j)_{AB})$. In the case where $i=s-1$, \Cref{lem:Hreg} implies that $H_{s-1}\in \Hc$. Let $\ell$ be the minimal index such that $X(H_{s-1}, \ell) = X(H_{s-1})$. By \Cref{lem:wrap2}, there exists a hyperedge $H'\in \Hc$ such that $[1,l-1]\subset H'$ and $X(H')=X(H_{s-1},(H_{s-1})_{AB})\in H_{s-1}^+$. The minimality of $S$ implies that $X(H_1)< X(H')$. Since $(H_s)_{AB}\leq X(H_s)<l\leq X(H_{s-1})$, it follows that $(H_s)_{AB},X(H_s)\in H'$ and thus $X(H_s)>X(H_1)$, which is a contradiction.

    We can now conclude the proof since assuming that the pseudo-join $X^{AB}$ is cyclic with a minimal cycle of length greater than two leads to a contradiction.
\end{proof}

\begin{lemma}\label{lem:twocycle-reg-cyc}
    Let $A$ and $B$ be two distincts acyclic orientations of $\HH$. If $X^{AB}$ is cyclic with a $2$-cycle given by some $H \in \Hc$ and $H' \in \Hr$, then $X(H) > X(H')$ and $X(H), X(H') \in H^-$.
\end{lemma}

\begin{proof}
    Suppose $X(H) = X(H, \ell)$ for some $\ell \ge H_{AB}$ and $(H_i, h_i)_{i = 1}^k$ is a sequence realizing it. Suppose also that $(H'_i, h'_i)_{i = 1}^s$ is a sequence realizing $X(H') = X(H', H'_{AB})$ (the equality is given by \Cref{lem:Hreg}). Note that $X$ having a $2$-cycle at $H$ and $H'$ implies that $X(H), X(H') \in H \cap H'$. 

    \emph{Case 1.} Suppose $X(H) < X(H')$. If $H'_{AB} > X(H)$, then $(H_i, h_i)_{i = 1}^k \|(H'_i, h'_i)_{i = 1}^s$ is an $(H, \ell)$-sequence, a contradiction. Therefore, $H'_{AB} \le X(H)$. However, this implies $X(H) \in [h'_j, h'_{j+1}] \subseteq H'_{j+1}$ for some $j \in [s-1]$ by Lemma \ref{lem:Hreg}(1). Thus $(H_i, h_i)_{i = 1}^k \|(H'_i, h'_i)_{i = j+1}^s$ is an $(H, \ell)$-sequence, again a contradiction.
    
    \emph{Case 2.} Suppose $X(H) > X(H')$. If $H_{AB} > X(H')$, then $(H'_i, h'_i)_{i = 1}^s \|(H, H_{AB})$ is an $(H', H'_{AB})$-sequence, a contradiction. Hence $H_{AB} \le X(H')$. 
    
    Consider $X(H, X(H')) \ge X(H)$. Let $(\Bar{H}_i, \Bar{h}_i)_{i = 1}^{r}$ be a sequence that realizes $X(H, X(H'))$. In fact, $X(H, X(H')) > X(H)$. Otherwise, $(H'_i, h'_i)_{i = 1}^s \| (\Bar{H}_i, \Bar{h}_i)_{i = 2}^r$ is an $(H', H'_{AB})$-sequence, a contradiction. Therefore, $X(H') < X(H) < X(H, X(H'))$. By argument of concatenation, there is no $j\in [r-1]$ such that $X(H)\in [\Bar{h}_j,\Bar{h}_{j+1})$. This implies that $\Bar{H}_{j+1}\in\Hc$. Suppose $X(H) \in H^+$ and $X(H') \in H^-$. Then $\Bar{h}_{j+1} = (\Bar{H}_{j+1})_{AB} \in H^+$ because $\Bar{h}_{j+1} > X(H)$. Moreover, $A(H), B(H) \le H_{AB} \le X(H') \le \Bar{h}_j$, so $A(H), B(H) \in (\Bar{H}_{j+1})^-$. Thus at least one of $A$ and $B$ is cyclic, a contradiction. The same argument also holds for the case when $X(H), X(H') \in H^+$. Therefore, we can conclude that $X(H), X(H') \in H^-$. 
\end{proof}

\begin{lemma}\label{lem:existencehj}
    Let $A$ and $B$ be two distinct acyclic orientations of $\HH$. Let $H \in \Hc$ and $H' \in \HH$. Let $(\tilde{H}_i, \tilde{h}_i)_{i = 1}^r$ be a sequence realizing $X(H, \ell') > X(H) = X(H, \ell)$ for some $H_{AB} \le \ell' < X(H)$. Moreover, let $L \in \HH$ be such that $\ell \in L$ and $L_{AB} > \ell$. If $X(H, \ell') \in H^+$, then $L_{AB} \notin \tilde{H}_i$ for any $i \in [2, r]$ such that $X(H) < \tilde{h}_i$. 
\end{lemma}

\begin{proof}
Suppose for contradiction that $L_{AB} \in \tilde{H}_j$ for some $j \in [2, r]$ such that $X(H) < \tilde{h}_j$. We check the following cases and refer the reader to \Cref{lemma} for an instance of each of the cases:
    
    \emph{Case 1}. If $L_{AB} < \tilde{h}_j$, then $(H, \ell) \| (L, L_{AB}) \| (\tilde{H}_i, \tilde{h}_i)_{i = j}^r$ is an $(H, \ell)$-sequence, a contradiction.

    \emph{Case 2}. If $L_{AB} = \tilde{h}_j \ne \tilde{h}_r$, then $(H, \ell) \| (L, L_{AB}) \| (\tilde{H}_i, \tilde{h}_i)_{i = j+1}^r$ is an $(H, \ell)$-sequence, a contradiction. Since $\tilde{h}_j \ne \tilde{h}_r$, the last sequence in the concatenation is nonempty.

    \emph{Case 3}. If $L_{AB} > \tilde{h}_j \ne \tilde{h}_r$, then $L_{AB} \in [\tilde{h}_p, \tilde{h}_{p+1})$ for some $p \in [j, r-1]$. If $L_{AB} \in \tilde{H}_{p+1}$, then $(H, \ell) \| (L, L_{AB}) \| (\tilde{H}_i, \tilde{h}_i)_{i = p+1}^r$ is an $(H, \ell)$-sequence. 
If $L_{AB} \notin \tilde{H}_{p+1}$, then $X(H) < \tilde{h}_j \le \tilde{h}_p$ is in $(\tilde{H}_{p+1})_-$. Then $(H_i, h_i)_{i=1}^k\|(\tilde{H}_i, \tilde{h}_i)_{i = p+1}^r$ is an $(H, \ell)$-sequence. Both sequences contradict the construction of $X(H)$.

    \emph{Case 4}. If $L_{AB} \ge \tilde{h}_j = \tilde{h}_r$, then $L_{AB} \in H_+$ and $(H, \ell) \| (L, L_{AB})$ is an $(H, \ell)$-sequence, a contradiction.
    \end{proof}

\begin{figure}[H]
 \scalebox{0.95}{\begin{tikzpicture}
 
 \node at (-5.7,0) {\begin{tikzpicture}[scale=0.35, every node/.style={font=\small}]
\def\R{3}
\def\RR{2.5}
\def\RRR{3.5}
\node (v1) at ({\RR*cos(60)},{\RR*sin(60)}) {3};
\node (v2) at ({\RR*cos(360)},{\RR*sin(360)}) {2};
\node (v3) at ({\RR*cos(120)},{\RR*sin(120)}) {4};
\node  at ({\RR*cos(180)},{\RR*sin(180)}) {5};
\node  at ({\RR*cos(240)},{\RR*sin(240)}) {6};
\node  at ({\RR*cos(300)},{\RR*sin(300)}) {1};
\draw[dashed,gray] (0,0) circle (\R);
\draw[line width=1.5pt, color=Black] (0,0) ++(240:3.5) arc[start angle=240, end angle=420, radius=3.5];
\draw[line width=1.5pt,color=Black!75] (0,0) ++(60:5) arc[start angle=60, end angle=120, radius=5];
\draw[line width=1.5pt,color=Black!50] (0,0) ++(120:4) arc[start angle=120, end angle=360, radius=4];
\draw[line width=1.5pt,color=Black!25] (0,0) ++(120:4.5) arc[start angle=120, end angle=240, radius=4.5];

\node at ({4*cos(360)},{4*sin(360)}) {\large{$\textcolor{blue!30}{\bullet}$}};
\node at ({3.5*cos(360)},{3.5*sin(360)}) {\large{$\textcolor{blue!30}{\bullet}$}};
\node at ({3.5*cos(300)},{3.5*sin(300)}) {\large{$\textcolor{red!70}{\bullet}$}};
\node at ({4.5*cos(180)},{4.5*sin(180)}) {\large{$\textcolor{red!70}{\bullet}$}};
\node at ({4*cos(180)},{4*sin(180)}) {\large{$\textcolor{red!70}{\bullet}$}};
\node at ({4.5*cos(240)},{4.5*sin(240)}) {\large{$\textcolor{blue!30}{\bullet}$}};
\node at ({5*cos(120)},{5*sin(120)}) {\large{$\textcolor{blue!30}{\bullet}$}};
\node at ({5*cos(60)},{5*sin(60)}) {\large{$\textcolor{red!70}{\bullet}$}};

\node at (0,6.5) {\emph{Case 1}};
\node at (0,-7) {$\begin{matrix*}[l]
    \textcolor{Black}{\Tilde{H}_1=1236} & L=\textcolor{Black!75}{34}  \\ \Tilde{H}_2=\textcolor{Black!50}{12456} & \\ \Tilde{H}_3=\textcolor{Black!25}{456}
\end{matrix*}$};

\end{tikzpicture}};

\node at (-1.85,0) {\begin{tikzpicture}[scale=0.35, every node/.style={font=\small}]
\def\R{3}
\def\RR{2.5}
\def\RRR{3.5}
\node (v1) at ({\RR*cos(60)},{\RR*sin(60)}) {3};
\node (v2) at ({\RR*cos(360)},{\RR*sin(360)}) {2};
\node (v3) at ({\RR*cos(120)},{\RR*sin(120)}) {4};
\node  at ({\RR*cos(180)},{\RR*sin(180)}) {5};
\node  at ({\RR*cos(240)},{\RR*sin(240)}) {6};
\node  at ({\RR*cos(300)},{\RR*sin(300)}) {1};
\draw[dashed,gray] (0,0) circle (\R);
\draw[line width=1.5pt, color=Black] (0,0) ++(240:3.5) arc[start angle=240, end angle=420, radius=3.5];
\draw[line width=1.5pt,color=Black!75] (0,0) ++(60:5) arc[start angle=60, end angle=180, radius=5];
\draw[line width=1.5pt,color=Black!50] (0,0) ++(120:4) arc[start angle=120, end angle=360, radius=4];
\draw[line width=1.5pt,color=Black!25] (0,0) ++(120:4.5) arc[start angle=120, end angle=240, radius=4.5];

\node at ({4*cos(360)},{4*sin(360)}) {\large{$\textcolor{blue!30}{\bullet}$}};
\node at ({3.5*cos(360)},{3.5*sin(360)}) {\large{$\textcolor{blue!30}{\bullet}$}};
\node at ({3.5*cos(300)},{3.5*sin(300)}) {\large{$\textcolor{red!70}{\bullet}$}};
\node at ({4.5*cos(120)},{4.5*sin(120)}) {\large{$\textcolor{red!70}{\bullet}$}};
\node at ({4*cos(120)},{4*sin(120)}) {\large{$\textcolor{red!70}{\bullet}$}};
\node at ({4.5*cos(240)},{4.5*sin(240)}) {\large{$\textcolor{blue!30}{\bullet}$}};
\node at ({5*cos(120)},{5*sin(120)}) {\large{$\textcolor{red!70}{\bullet}$}};
\node at ({5*cos(60)},{5*sin(60)}) {\large{$\textcolor{blue!30}{\bullet}$}};

\node at (0,6.5) {\emph{Case 2}};
\node at (0,-7) {$\begin{matrix*}[l]
    \textcolor{Black}{\Tilde{H}_1=1236} & L=\textcolor{Black!75}{345}  \\ \Tilde{H}_2=\textcolor{Black!50}{12456} & \\ \Tilde{H}_3=\textcolor{Black!25}{456}
\end{matrix*}$};

\end{tikzpicture}};

\node at (1.85,0) {\begin{tikzpicture}[scale=0.35, every node/.style={font=\small}]
\def\R{3}
\def\RR{2.5}
\def\RRR{3.5}
\node (v1) at ({\RR*cos(60)},{\RR*sin(60)}) {3};
\node (v2) at ({\RR*cos(360)},{\RR*sin(360)}) {2};
\node (v3) at ({\RR*cos(120)},{\RR*sin(120)}) {4};
\node  at ({\RR*cos(180)},{\RR*sin(180)}) {5};
\node  at ({\RR*cos(240)},{\RR*sin(240)}) {6};
\node  at ({\RR*cos(300)},{\RR*sin(300)}) {1};
\draw[dashed,gray] (0,0) circle (\R);
\draw[line width=1.5pt, color=Black] (0,0) ++(240:3.5) arc[start angle=240, end angle=420, radius=3.5];
\draw[line width=1.5pt,color=Black!75] (0,0) ++(60:5) arc[start angle=60, end angle=180, radius=5];
\draw[line width=1.5pt,color=Black!50] (0,0) ++(120:4) arc[start angle=120, end angle=360, radius=4];
\draw[line width=1.5pt,color=Black!25] (0,0) ++(120:4.5) arc[start angle=120, end angle=240, radius=4.5];

\node at ({4*cos(360)},{4*sin(360)}) {\large{$\textcolor{blue!30}{\bullet}$}};
\node at ({3.5*cos(360)},{3.5*sin(360)}) {\large{$\textcolor{blue!30}{\bullet}$}};
\node at ({3.5*cos(300)},{3.5*sin(300)}) {\large{$\textcolor{red!70}{\bullet}$}};
\node at ({4.5*cos(120)},{4.5*sin(120)}) {\large{$\textcolor{red!70}{\bullet}$}};
\node at ({4*cos(120)},{4*sin(120)}) {\large{$\textcolor{red!70}{\bullet}$}};
\node at ({4.5*cos(240)},{4.5*sin(240)}) {\large{$\textcolor{blue!30}{\bullet}$}};
\node at ({5*cos(120)},{5*sin(120)}) {\large{$\textcolor{red!70}{\bullet}$}};
\node at ({5*cos(180)},{5*sin(180)}) {\large{$\textcolor{blue!30}{\bullet}$}};

\node at (0,6.5) {\emph{Case 3}};
\node at (0,-7) {$\begin{matrix*}[l]
    \textcolor{Black}{\Tilde{H}_1=1236} & L=\textcolor{Black!75}{345}  \\ \Tilde{H}_2=\textcolor{Black!50}{12456} & \\ \Tilde{H}_3=\textcolor{Black!25}{456}
\end{matrix*}$};

\end{tikzpicture}};

\node at (5.7,0) {\begin{tikzpicture}[scale=0.35, every node/.style={font=\small}]
\def\R{3}
\def\RR{2.5}
\def\RRR{3.5}
\node (v1) at ({\RR*cos(60)},{\RR*sin(60)}) {3};
\node (v2) at ({\RR*cos(360)},{\RR*sin(360)}) {2};
\node (v3) at ({\RR*cos(120)},{\RR*sin(120)}) {4};
\node  at ({\RR*cos(180)},{\RR*sin(180)}) {5};
\node  at ({\RR*cos(240)},{\RR*sin(240)}) {6};
\node  at ({\RR*cos(300)},{\RR*sin(300)}) {1};
\draw[dashed,gray] (0,0) circle (\R);
\draw[line width=1.5pt, color=Black] (0,0) ++(240:3.5) arc[start angle=240, end angle=420, radius=3.5];
\draw[line width=1.5pt,color=Black!75] (0,0) ++(60:5) arc[start angle=60, end angle=240, radius=5];
\draw[line width=1.5pt,color=Black!50] (0,0) ++(120:4) arc[start angle=120, end angle=360, radius=4];
\draw[line width=1.5pt,color=Black!25] (0,0) ++(120:4.5) arc[start angle=120, end angle=240, radius=4.5];

\node at ({4*cos(360)},{4*sin(360)}) {\large{$\textcolor{blue!30}{\bullet}$}};
\node at ({3.5*cos(360)},{3.5*sin(360)}) {\large{$\textcolor{blue!30}{\bullet}$}};
\node at ({3.5*cos(300)},{3.5*sin(300)}) {\large{$\textcolor{red!70}{\bullet}$}};
\node at ({4.5*cos(120)},{4.5*sin(120)}) {\large{$\textcolor{red!70}{\bullet}$}};
\node at ({4*cos(120)},{4*sin(120)}) {\large{$\textcolor{red!70}{\bullet}$}};
\node at ({4.5*cos(240)},{4.5*sin(240)}) {\large{$\textcolor{blue!30}{\bullet}$}};
\node at ({5*cos(120)},{5*sin(120)}) {\large{$\textcolor{red!70}{\bullet}$}};
\node at ({5*cos(240)},{5*sin(240)}) {\large{$\textcolor{blue!30}{\bullet}$}};

\node at (0,6.5) {\emph{Case 4}};
\node at (0,-7) {$\begin{matrix*}[l]
    \textcolor{Black}{\Tilde{H}_1=1236} & L=\textcolor{Black!75}{3456}  \\ \Tilde{H}_2=\textcolor{Black!50}{12456} & \\ \Tilde{H}_3=\textcolor{Black!25}{456}
\end{matrix*}$};

\end{tikzpicture}};

\end{tikzpicture}}
\caption{Four hypergraphs $\HH$ with the juxtaposed acyclic orientations $A$ and $B$ with sources respectively represented with $\textcolor{blue!30}{\bullet}$ and $\textcolor{red!70}{\bullet}$. Those four hypergraphs include one instance of each case discussed in the proof of \Cref{lem:existencehj}. }
\label{lemma}\end{figure}

\begin{proposition}\label{prop:twocycle-reg-cyc}
    Let $A$ and $B$ be two acyclic orientations of $\HH$. If $X^{AB}$ is cyclic with a $2$-cycle given by some $H \in \Hc$ and $H' \in \Hr$, then there exists an interval $D \subseteq [n]$ such that $\HH|_D$ has a hugging quadruple without a fix.
\end{proposition}

\begin{proof}
We keep the same notation as in the proof of \Cref{lem:twocycle-reg-cyc} and we now require $\ell$ to be the smallest index in $H$ such that $X(H, \ell) = X(H)$. From the proof of \Cref{lem:twocycle-reg-cyc} we have that $H_{AB}\leq X(H')<X(H)<X(H,X(H'))$ which implies that $H_{AB}<\ell$. Hence, Let $(\tilde{H}_i, \tilde{h}_i)_{i = 1}^r$ be a minimal sequence realizing $X(H, \ell - 1)$.  Minimality of the sequence here means that $\tilde{h}_i\notin \tilde{H}_{i+2}$ for $i\in [r-2]$. By minimality of $\ell$ we have that  $\tilde{H}_2 \in \Hc$ with $\ell - 1 = \tilde{h}_1 = \max (\tilde{H}_2^-)$ and $\tilde{h}_2 \in (\tilde{H}_2^+)$. The acyclicity of $A$ and $B$ gives us that $X(H) < \tilde{h}_2 < \min(H^+)$ which also implies that $[X(H'),X(H)]\subset H$. Arguing by concatenation, we can show that $X(H') \le \ell - 1$.  To summarize, the following relations hold and refer the reader to \Cref{fig:prop5.9.1}:

\[
H_{AB} \le X(H') \le \ell - 1 =  \tilde{h}_1 < \ell \le X(H) < \tilde{h}_2 < \min(H^+) \le \tilde{h}_r = X(H, \ell - 1).
\]

\begin{figure}[H]
 \scalebox{1}{\begin{tikzpicture}
 
\def\R{1.8}
\def\RR{1.25}
\node (v1) at ({3.9*cos(280)},{3.9*sin(280)}) {$1$};
\draw[dashed] ({\R*cos(280)},{\R*sin(280)}) -- ({3.5*cos(280)},{3.5*sin(280)});
\node (v2) at ({3.9*cos(260)},{3.9*sin(260)}) {$n$};
\draw[dashed] ({\R*cos(260)},{\R*sin(260)}) -- ({3.5*cos(260)},{3.5*sin(260)});
\node (v3) at ({(3.9*cos(350))+0.3},{(3.9*sin(350))}) {$\ell-1=\Tilde{h}_1$};
\draw[dashed] ({\R*cos(350)},{\R*sin(350)}) -- ({3.5*cos(350)},{3.5*sin(350)});
\node (v4) at ({3.9*cos(360)},{3.9*sin(360)}) {$\ell$};
\draw[dashed] ({\R*cos(360)},{\R*sin(360)}) -- ({3.5*cos(360)},{3.5*sin(360)});

\draw[dashed,gray] (0,0) circle (\R);

\draw[line width=1.75pt, draw=red!70] (0,0) ++(220:2.75) arc[start angle=220, end angle=400, radius=2.75];
\node at ({2.75*cos(380)},{2.75*sin(380)}) {$\bullet$};
\node at ({3.9*cos(380)},{3.9*sin(380)}) {$X(H)$};
\draw[dashed] ({\R*cos(380)},{\R*sin(380)}) -- ({3.5*cos(380)},{3.5*sin(380)});

\node at ({2.75*cos(320)},{2.75*sin(320)}) {$\bullet$};
\node at ({3.9*cos(320)},{3.9*sin(320)}) {$H_{AB}$};
\draw[dashed] ({\R*cos(320)},{\R*sin(320)}) -- ({3.5*cos(320)},{3.5*sin(320)});

\node at ({2.75*cos(350)},{2.75*sin(350)}) {$\bullet$};


 \draw[line width=1.75pt, draw=blue!30] (0,0) ++(300:2.5) arc[start angle=300, end angle=390, radius=2.5];

 \node at ({2.5*cos(330)},{2.5*sin(330)}) {$\bullet$};
\node at ({3.9*cos(330)},{3.9*sin(330)}) {$X(H')$};
\draw[dashed] ({\R*cos(330)},{\R*sin(330)}) -- ({3.5*cos(330)},{3.5*sin(330)});


\draw[line width=1.75pt, draw=yellow] (0,0) ++(180:2.25) arc[start angle=180, end angle=350, radius=2.25];
\node at ({2.25*cos(190)},{2.25*sin(190)}) {$\bullet$};
\node at ({3.9*cos(190)},{3.9*sin(190)}) {$\Tilde{h}_2$};
\draw[dashed] ({\R*cos(190)},{\R*sin(190)}) -- ({3.5*cos(190)},{3.5*sin(190)});

 
\draw[line width=1.75pt, draw=orange!70] (0,0) ++(210:2.5) arc[start angle=210, end angle=250, radius=2.5];
\node at ({2.5*cos(240)},{2.5*sin(240)}) {$\bullet$};
\node at ({(3.9*cos(240))-0.2},{3.9*sin(240)}) {$\Tilde{h}_r=X(H,\ell-1)$};
\draw[dashed] ({\R*cos(240)},{\R*sin(240)}) -- ({3.5*cos(240)},{3.5*sin(240)});


\node at (0,0.75) {$H := \color{red!70}{\bullet}$};
\node at (0,0.25) {$H' := \color{blue!30}{\bullet}$};
\node at (0,-0.25) {$\Tilde{H}_2 := \color{yellow}{\bullet}$};
\node at (0,-0.75) {$\Tilde{H}_r := \color{orange!70}{\bullet}$};

\end{tikzpicture}}
 \caption{Example of hyperedges illustrating the information at the beginning of the proof of \Cref{lem:twocycle-reg-cyc}}\label{fig:prop5.9.1}
\end{figure}
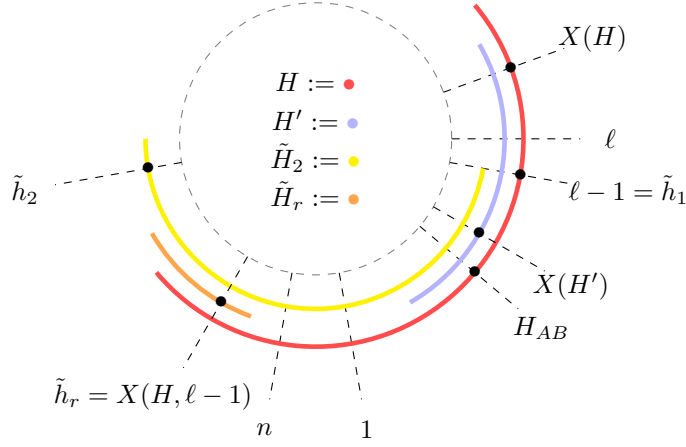

We first show that there exists an interval $D^1 \subseteq [n]$ such that $\HH|_{D^1}$ has a hugging quadruple. Let $D^1 = [\ell - 1, \tilde{h}_r]$. 

    \begin{itemize}
        \item Let $I^1 = H'\cap D^1$. First, $\mathbf{\ell - 1, \ell} \in [X(H'), X(H)] \subseteq \mathbf{H'}$. Secondly, $\tilde{h}_r \notin H'$, otherwise $(H'_i, h'_i)_{i = 1}^s\| (\tilde{H}_i, \tilde{h}_i)_{i = 1}^r$ is an $(H', H'_{AB})$-sequence. With $H' \in \Hr$, we conclude that $\mathbf{I^1 \in {(\HH|_{D^1})}_{\reg}}$.
        
        \item Let $\Ic^1 = H\cap D^1$. Similarly, $\mathbf{\ell - 1, \ell \in} [X(H'), X(H)] \subseteq \mathbf{H}$. Moreover, $\tilde{h}_r \in H$, so $\mathbf{\Ic^1 \in {(\HH|_{D^1})}_{\cyc}}$.

        \item Let $J^1 = \tilde{H}_r\cap D^1$. Immediately $\mathbf{\tilde{h}_r \in J^1}$. This implies $\mathbf{\tilde{h}_r - 1 \in J^1}$. By acyclicity of $A$ and $B$, we have that $H_{AB}\notin \tilde{H}_r$, hence $\ell - 1 \notin \tilde{H}_r$, so $\mathbf{J^1 \in {(\HH|_{D^1})}_{\reg}}$.

        \item Let $\Jc^1 = \tilde{H}_2\cap D^1$. Since $\tilde{h}_2 < \tilde{h}_r$ and $\tilde{h}_2 \in \tilde{H}_2^+$, we have $\mathbf{\tilde{h}_r-1, \tilde{h}_r \in \Jc^1}$. Recall also $\ell - 1 = \tilde{h}_2 \in \tilde{H}_2$, so $\mathbf{\Jc^1 \in {(\HH|_{D^1})}_{\cyc}}$. 
    \end{itemize}
Therefore, $I^1, \Ic^1, J^1, \Jc^1$ is a hugging quadruple of $\HH|_{D^1}$ (the bold texts show that they satify the conditions in \Cref{def:hug}). If this quadruple does not have a fix in $\HH|_{D^1}$, then we are done. Now suppose there exists an $L^1 \in \HH$ such that $L^1\cap D^1$ is a fix of the quadruple. Namely, 
\[
\{\ell, \tilde{h}_r - 1\} \subseteq L^1\cap D^1 \subseteq I^1 \cup \Ic^1 \cup J^1 \cup \Jc^1.
\]
This also implies that 
\[
\{\ell, \tilde{h}_r - 1\} \subseteq L^1 \subseteq  H' \cup H \cup \tilde{H}_r \cup \tilde{H}_2.
\]

\begin{claim}\label{clm:cyclicfix}
    $\tilde{h}_2 \notin L^1$. In particular, the interval $[\ell, \tilde{h}_r - 1]$ is not contained in $L^1$, so $L^1 \in \Hc$ and $\ell - 1, \ell \in (L^1)^-$ and $\tilde{h}_r - 1, \tilde{h}_r \in (L^1)^+$. 
\end{claim}
\begin{proof}[Proof of \Cref{clm:cyclicfix}]
Suppose for contradiction that $\tilde{h}_2 \in L^1$ and assume without loss of generality that $\tilde{h}_2  = A(\tilde{H}_2)$. We discuss two cases based on where $L^1_{AB}$ lies. 

\emph{Case 1.} 
Suppose $L^1_{AB} \le \ell$. Then $A(L^1), B(L^1) \in H^-$. Recall that $A(H), B(H) \le H_{AB} \le \tilde{h}_1$, so $A(H), B(H) \in \tilde{H}_2^-$. Then the orientation $A$ gives
\[
A(H)\xrightarrow{H}A(L^1)\xrightarrow{L^1}A(\tilde{H}_2)\xrightarrow{\tilde{H}_2} A(H).
\]
Thus $A$ is cyclic (even if $A(H) = A(L^1)$), a contradiction. 

\emph{Case 2.} Suppose $L^1_{AB} > \ell$. Since $\ell\in L^1$, by \Cref{lem:existencehj} we have that $L^1_{AB} \notin \tilde{H}_2 \cup \tilde{H}_r$. If $L^1_{AB} \in H'$, then $A(H')\in L^1$ since $L^1_{AB}\notin \tilde{H}_2$ and $\{\ell, \tilde{h}_r - 1\} \subseteq L^1 \subseteq  H' \cup H \cup \tilde{H}_r \cup \tilde{H}_2$. Therefore, the orientation $A$ gives 
        
    $$A(H')\xrightarrow{H'}A(L^1)\xrightarrow{L^1}A(\tilde{H}_2)\xrightarrow{\tilde{H}_2} A(H').$$

Thus $A$ is cyclic (even if $A(H') = A(L^1)$), a contradiction. If $L^1_{AB} \in H$, we use the same above reasoning to get that $A$ is cyclic, contradiction.
\end{proof}
By \Cref{clm:cyclicfix}, the following inequality holds and refer the reader to \Cref{fig:prop5.9.2}:
\[
\tilde{h}_2 < \min((L^1)^+) \le \tilde{h}_r -1 < \tilde{h}_r.
\]

\begin{figure}[H]
 \scalebox{1}{\begin{tikzpicture}
 
\def\R{1.8}
\def\RR{1.25}
\node (v1) at ({3.9*cos(280)},{3.9*sin(280)}) {$1$};
\draw[dashed] ({\R*cos(280)},{\R*sin(280)}) -- ({3.5*cos(280)},{3.5*sin(280)});
\node (v2) at ({3.9*cos(260)},{3.9*sin(260)}) {$n$};
\draw[dashed] ({\R*cos(260)},{\R*sin(260)}) -- ({3.5*cos(260)},{3.5*sin(260)});
\node (v3) at ({(3.9*cos(350))+0.3},{(3.9*sin(350))}) {$\ell-1$};
\draw[dashed] ({\R*cos(350)},{\R*sin(350)}) -- ({3.5*cos(350)},{3.5*sin(350)});
\node (v4) at ({3.9*cos(360)},{3.9*sin(360)}) {$\ell$};
\draw[dashed] ({\R*cos(360)},{\R*sin(360)}) -- ({3.5*cos(360)},{3.5*sin(360)});

\draw[dashed,gray] (0,0) circle (\R);

\draw[line width=1.75pt, draw=red!70] (0,0) ++(220:2.75) arc[start angle=220, end angle=400, radius=2.75];



 \draw[line width=1.75pt, draw=blue!30] (0,0) ++(300:2.5) arc[start angle=300, end angle=390, radius=2.5];


\draw[line width=1.75pt, draw=yellow] (0,0) ++(180:2.25) arc[start angle=180, end angle=350, radius=2.25];
\node at ({3.9*cos(190)},{3.9*sin(190)}) {$\Tilde{h}_2$};
\draw[dashed] ({\R*cos(190)},{\R*sin(190)}) -- ({3.5*cos(190)},{3.5*sin(190)});

 
\draw[line width=1.75pt, draw=orange!70] (0,0) ++(210:2.5) arc[start angle=210, end angle=250, radius=2.5];
\node at ({(3.9*cos(240))-0.2},{3.9*sin(240)}) {$\Tilde{h}_r$};
\draw[dashed] ({\R*cos(240)},{\R*sin(240)}) -- ({3.5*cos(240)},{3.5*sin(240)});

\node at ({(3.9*cos(230))-0.2},{3.9*sin(230)}) {$\Tilde{h}_{r-1}$};
\draw[dashed] ({\R*cos(230)},{\R*sin(230)}) -- ({3.5*cos(230)},{3.5*sin(230)});


\draw[line width=1.75pt, draw=gray] (0,0) ++(200:3) arc[start angle=200, end angle=370, radius=3];


\node at (0,1) {$H := \color{red!70}{\bullet}$};
\node at (0,0.5) {$H' := \color{blue!30}{\bullet}$};
\node at (0,0) {$\Tilde{H}_2 := \color{yellow}{\bullet}$};
\node at (0,-0.5) {$\Tilde{H}_r := \color{orange!70}{\bullet}$};
\node at (0,-1) {$L^1 := \color{gray}{\bullet}$};

\node at (-1.6,1.6) {$D^1$};


\draw[dashed,line width=1.75pt, draw=black] (0,0) ++(-10:1.8) arc[start angle=-10, end angle=240, radius=1.8];

\end{tikzpicture}}
 \caption{Example of hyperedges illustrating the information given by \Cref{clm:cyclicfix} in the proof of \Cref{lem:twocycle-reg-cyc}.}\label{fig:prop5.9.2}
\end{figure}
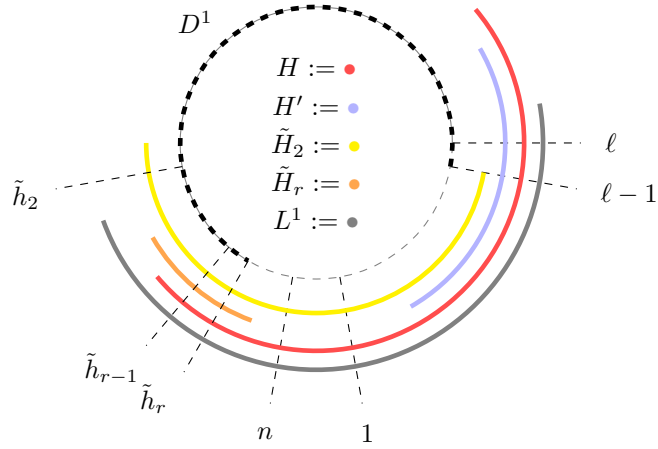

Let $D^2 = [\ell - 1, \min((L^1)^+)] \subset D^1$ and consider $\HH|_{D^2}$. We find a new hugging quadruple as follows.

\begin{itemize}
    \item Let $I^2 = H' \cap D^2$. As before, $\mathbf{\ell - 1, \ell \in} [X(H'), X(H)] \subseteq \mathbf{H'}$. Secondly, $\min((L^1)^+) \notin H'$, otherwise $\tilde{h}_2 \in [X(H'), \min((L^1)^+)] \subseteq H'$, causing at least one of $A$ and $B$ to be cyclic. Therefore, together with $H' \in \Hr$, we conclude that $\mathbf{I^2 \in {(\HH|_{D^2})}_{\reg}}$.
    \item Let $\Ic^2 = L^1 \cap D^2$. By \Cref{clm:cyclicfix}, $\mathbf{\ell - 1, \ell \in L}$. Trivially $\min((L^1)^+) \in L^1$, so $\mathbf{\Ic^2 \in {(\HH|_{D^2})}_{\cyc}}$. 
    
    \item Let $J^2 = \tilde{H}_{p^1+1} \cap D^2$, where $p^1 \in [2, r-1]$ such that $\min((L^1)^+) \in (\tilde{h}_{p^1}, \tilde{h}_{p^1+1}]$. Moreover, $[\tilde{h}_{p^1}, \tilde{h}_{p^1+1}] \subseteq \tilde{H}_{p^1+1}$. This follows from the fact that $X(H) < \tilde{h}_{p^1}$ and $X(H) \in (\tilde{H}_{p^1+1})_-$ would extend the $(H, \ell)$-sequence. Therefore, $\mathbf{\min((L^1)^+) - 1 \in \tilde{H}_{p^1+1}}$ and $\mathbf{\min((L^1)^+) \in \tilde{H}_{p^1+1}}$. The minimality of $(\tilde{H}_i, \tilde{h}_i)_{i = 1}^r$ implies that $\ell - 1 \notin \tilde{H}_{p^1+1}$ and thus $\mathbf{J^2 \in {(\HH|_{D^2})}_{\reg}}$.

    \item Let $\Jc^2 = \tilde{H}_2 \cap D^2$. As before, since $\ell - 1 \in (\tilde{H}_2)^-$, $\tilde{h}_2 \in (\tilde{H}_2)^+$, and $\tilde{h}_2 < \min((L^1)^+)$, it follows that $\mathbf{\min((L^1)^+), \min((L^1)^+)-1 \in \Jc^2}$ and $\mathbf{\Jc^2 \in {(\HH|_{D^2})}_{\cyc}}$.
\end{itemize}

Therefore, $I^2, \Ic^2, J^2, \Jc^2$ is a hugging quadruple of $\HH|_{D^2}$. If this quadruple does not have a fix in $\HH|_{D^2}$, then again we are done. 

Now suppose there exists an $L^2 \in \HH$ such that $L^2\cap D^2$ is a fix of the quadruple. Namely, 
\[
\{\ell, \min((L^1)^+) - 1\} \subseteq L^2\cap D^2 \subseteq I^2 \cup \Ic^2 \cup J^2 \cup \Jc^2.
\]
This also implies that 
\[
\{\ell, \min((L^1)^+) - 1\} \subseteq L^2 \subseteq  H' \cup L^1 \cup \tilde{H}_{p^1+1} \cup \tilde{H}_2 \subseteq H' \cup H \cup \tilde{H}_r \cup \tilde{H}_2 \cup \tilde{H}_{p^1+1}.
\]
We want to show that $[\ell, \min((L^1)^+) - 1] \nsubseteq L^2$. Suppose otherwise for contradiction. This implies that $\tilde{h}_2\in L^2$ and following the proof of \Cref{clm:cyclicfix} we get a contradiction (the additional case that $L^2_{AB} > \ell$ and $L^2_{AB} \in \tilde{H}_{p^1+1}$ is ruled out by \Cref{lem:existencehj}).

Restrict $\HH$ repeatedly and find new hugging quadruples as defined above. If $L^i$ is a fix for the hugging quadruple $I^i, \Ic^i, J^i, \Jc^i$, then take $D^{i+1} = [\ell - 1, \min((L^i)^+)]$, and 
\begin{itemize}
    \item $I^{i+1} = H' \cap D^{i+1}$, 
    \item $\Ic^{i+1} = L^i \cap D^{i+1}$, 
    \item $J^{i+1} = \tilde{H}_{p^i+1} \cap D^{i+1}$, where $p^i \in [2, p^{i-1}-1]$ such that $\min((L^i)^+) \in (\tilde{h}_{p^i}, \tilde{h}_{p^i+1}]$.
    \item $\Jc^{i+1} = \tilde{H}_2 \cap D^{i+1}$.
\end{itemize}

By construction, $(L^i)^+\subset (L^{i+1})^+$, which implies that $L^i\neq L^{i+1}$. Since $\tilde{H}_{p^i+1}$ cannot contain $\tilde{h}_2$ and $\HH$ is finite this procedure must terminate. Therefore, there must exist a hugging quadruple without a fix, as desired.
\end{proof}

\begin{proposition}\label{prop:twocycle}
    Let $A$ and $B$ be two acyclic orientations of $\HH$. If $X^{AB}$ is cyclic with a $2$-cycle given by some $H, H' \in \HH$, then there exists an interval $D \subseteq [n]$ such that $\HH|_D$ has a hugging quadruple without a fix.
\end{proposition}

\begin{proof}
\Cref{prop:twocycle-reg-cyc} covers the case when $H \in \Hc$ and $H' \in \Hr$, while \Cref{lem:Hreg} covers the case when $H, H' \in \Hr$.

Therefore, suppose $H, H' \in \Hc$. Assume without loss of generality that $X(H) > X(H')$ and consider the sequences $(H_i, h_i)_{i = 1}^k$ and $(H'_i, h'_i)_{i = 1}^s$ respectively realizing $X(H) = X(H, \ell)$ and $X(H') = X(H', \ell')$. By hypothesis, $X(H), X(H') \in H \cap H'$. 

\emph{Case 1}. Suppose $X(H) \in H'^+$. 

\begin{itemize}
    \item Suppose $H_{AB} \le X(H')$. Hence, $X(H, X(H')) \ge X(H)$, so $X(H, X(H')) \in H'^+$. This means we can extend $(H'_i, h'_i)_{i = 1}^s$ and since $X(H, X(H')) > X(H')$, we get a contradiction.

    \item Suppose $H_{AB} > X(H')$. Since $X(H, H_{AB}) \ge X(H)$ implies that $X(H, H_{AB}) \in H'^+$ and since $X(H, H_{AB}) > X(H')$, we can extend $(H'_i, h'_i)_{i = 1}^s$, contradiction. 
\end{itemize}

\emph{Case 2}. Suppose $X(H) \in H'^-$, then automatically $X(H') \in H'^-$. 
    \begin{itemize}
        \item Suppose $X(H) \in H^+$ and $X(H') \in H^-$ or both $X(H), X(H') \in H^+$. Then we reach a contradiction by the same argument as in \emph{Case 2.} of \Cref{lem:twocycle-reg-cyc}.
        \item Suppose $X(H), X(H') \in H^-$. To avoid extending  $(H'_i, h'_i)_{i = 1}^s$ with $(H, H_{AB})$, we have $H_{AB} \le X(H')$ and $X(H, X(H')) > X(H)$. We follow the notation in the proof of \Cref{prop:twocycle-reg-cyc}. As in the first bullet point of that proof, $\tilde{h}_r \notin H'$, otherwise we can extend $(H'_i, h'_i)_{i = 1}^s$. In particular, $\tilde{h}_r < \min(H'^+)$. We can conclude that $\mathbf{I^1} = H'\cap D^1 \mathbf{\in {(\HH|_{D^1})}_{\reg}}$. We finish the argument by continuing with the rest of the proof of~\Cref{prop:twocycle-reg-cyc}.
    \end{itemize}
\end{proof}

\subsection{Join}\label{subsec:join}

In the previous sections, we have shown that the pseudo-join $X^{AB}$ of two acyclic orientations $A$ and $B$ is acyclic and therefore an element in our poset $P_\HH$. In this section, we conclude the proof of \Cref{thmcyclic} by showing that $X^{AB}$ is indeed the join of $A$ and $B$. 

\begin{proposition}\label{join}
    Let $A$ and $B$ be two incomparable acyclic orientations of $\HH$. If ${(\HH|_D)}_{\reg}$ is closed under intersection for $\HH|_D$ for any interval $D=[x, y]$ where $1\le x < y \le n$, then the pseudo-join $X^{AB}$ of $A$ and $B$ is the join of $A$ and $B$.
\end{proposition}

\begin{proof}
    By construction, $X > A, B$. Assume that the pseudo-join $X$ of $A$ and $B$ is not their join. Then there exists an acyclic orientation $X'$ of $\HH$ such that $X' > A, B$, and a hyperedge $H \in \HH$ such that $H_{AB} \leq X'(H) < X(H)$. 
    
    \emph{Case 1}.
        Suppose that there exists an hyperedge $\Bar{H}\in \HH$ such that $X'(H)\in \Bar{H}$, $X'(H)<\Bar{H}_{AB}$ and $[X'(H),\Bar{H}_{AB}] \subseteq H$. Since $A$ and $B$ are acyclic and $\Bar{H}_{AB} \in H$, we have that $\{A(H), B(H)\} \nsubseteq \Bar{H}$. Therefore, $X'(H)$ must be in $\Bar{H}^+$ if $\Bar{H} \in \Hc$.

        By abuse of notation, take $H^- = H$ and $H^+ = \emptyset$ when $H \in \Hr$, and $\Bar{H}^+ = \Bar{H}$ and $\Bar{H}^- = \emptyset$ when $\Bar{H} \in \Hr$. In the following, note that taking $\min,\max$ operation of empty sets is disregarded.
        
        Consider the case when $[X'(H), \Bar{H}_{AB}] \subseteq H^-$ and refer the reader to \Cref{fig:prop5.11} (left). Take $D = [n]\setminus (H^+ \cup \Bar{H}^-)$. This implies the following restrictions:

        \begin{align*}
            H|_D &= [\max(\max(\Bar{H}^-)+1, \min(H)), \max(H^-)]\subseteq H^-, \\
            \Bar{H}|_D &= [\min(\Bar{H}^+), \min(\max(\Bar{H}), \min(H^+)-1)] \subseteq \Bar{H}^+.
        \end{align*}
        
        Thus $H|_D, \Bar{H}|_D \in {(\HH|_D)}_{\reg}$. Let $H_\cap|_D:= H|_D \cap \Bar{H}|_D$. Since $X'(H)\in \Bar{H}^+$, it follows that $\Bar{H}_{AB}\in \Bar{H}^+$ and the acyclicity of $A$ and $B$ implies that $H_\cap|_D\neq H|_D$. Since $X'$ is acyclic and also $X'(H)\in \Bar{H}^+$ and $\Bar{H}_{AB}\in H^-$, it follows that $X'(H)<\Bar{H}_{AB}<X'(\Bar{H})$. This implies that $H_\cap|_D\neq \Bar{H}|_D$. Since ${(\HH|_D)}_{\reg}$ is closed under intersection, $H_\cap|_D \in \HH$.    
        
        Consider the case when $[X'(H), \Bar{H}_{AB}] \subseteq H^+$ and refer the reader to \Cref{fig:prop5.11} (right).  For the same reason as above, $X'(H)\in \Bar{H}^+$. This implies that $H\in \Hr$. Indeed, otherwise the fact that $X'(H)\in H^+,\Bar{H}^+$ and $X(\Bar{H})<\Bar{H}_{AB}$ would contradict the acyclicity of $X'$. Take $D = [n]\setminus \Bar{H}^-$. This leads to the following restrictions:

        \begin{align*}
            H|_D &= [\max(\max(\Bar{H}^+)+1, \min(H)), \max(H)]\subseteq H, \\
            \Bar{H}|_D &= \Bar{H}^+.
        \end{align*}
        
        Again, $H_\cap|_D:= H|_D \cap \Bar{H}|_D \in {(\HH|_D)}_{\reg}$. Since $X'(H)\in \Bar{H}^+$, it follows that $\Bar{H}_{AB}\in \Bar{H}^+$. Moreover, since $H\in \Hr$, the acyclicity of $A$ and $B$ implies that $H_\cap|_D\neq H|_D$. Since $X'$ is acyclic and also $X'(H)\in \Bar{H}^+$ and $\Bar{H}_{AB}\in H$, it follows that $X'(H)<\Bar{H}_{AB}<X'(\Bar{H})$. This implies that $H_\cap|_D\neq \Bar{H}|_D$. Since ${(\HH|_D)}_{\reg}$ is closed under intersection, $H_\cap|_D \in \HH$.

    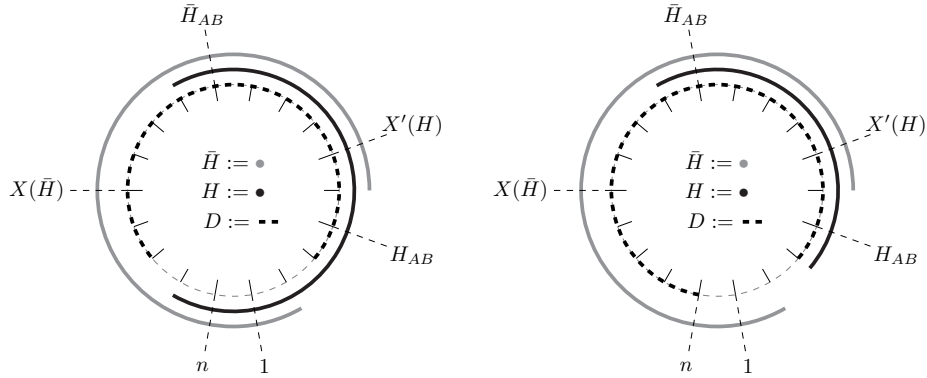
\begin{figure}[H]
 \scalebox{0.8}{\begin{tikzpicture}[decoration={markings, mark=at position 0.5 with {\arrow{>}}}]
     
 \node at (-4,0) {\begin{tikzpicture}
 
\def\R{1.75}
\def\RR{1.5}

\draw[dashed,gray] (0,0) circle (\R);


\draw[line width=1.75pt, draw=Black] (0,0) ++(240:2) arc[start angle=240, end angle=480, radius=2];

\draw[line width=1.75pt, draw=Black!50] (0,0) ++(0:2.25) arc[start angle=0, end angle=300, radius=2.25];


\draw[dashed] ({1.75*cos(280)},{1.75*sin(280)}) to ({2.7*cos(280)},{2.7*sin(280)});
\node at ({2.95*cos(280)},{2.95*sin(280)}){$1$};

\draw[dashed] ({1.75*cos(260)},{1.75*sin(260)}) to ({2.7*cos(260)},{2.7*sin(260)});
\node at ({2.95*cos(260)},{2.95*sin(260)}){$n$};

\draw[dashed] ({1.75*cos(340)},{1.75*sin(340)}) to ({2.7*cos(340)},{2.7*sin(340)});
\node at ({3.15*cos(340)},{3.15*sin(340)}){$H_{AB}$};

\draw[dashed] ({1.75*cos(20)},{1.75*sin(20)}) to ({2.7*cos(20)},{2.7*sin(20)});
\node at ({3.15*cos(20)},{3.15*sin(20)}){$X'(H)$};

\draw[dashed] ({1.75*cos(100)},{1.75*sin(100)}) to ({2.7*cos(100)},{2.7*sin(100)});
\node at ({2.95*cos(100)},{2.95*sin(100)}){$\Bar{H}_{AB}$};

\draw[dashed] ({1.75*cos(180)},{1.75*sin(180)}) to ({2.7*cos(180)},{2.7*sin(180)});
\node at ({3.25*cos(180)},{3.25*sin(180)}){$X(\Bar{H})$};
\newcommand{\hwplotB}{\raisebox{2pt}{\tikz{\draw[black,dashed,line width=1.75pt](0,0) -- (4mm,0);}}}
\node at (0,0) {$H := \color{Black}{\bullet}$};
\node at (0,0.5) {$\Bar{H} := \color{Black!50}{\bullet}$};
\node at (0.18,-0.5) {$D :=$ \hwplotB};

\draw[dashed,line width=1.75pt, draw=black] (0,0) ++(-40:1.75) arc[start angle=-40, end angle=220, radius=1.75];

\draw[] ({1.75*cos(280)},{1.75*sin(280)}) to ({1.5*cos(280)},{1.5*sin(280)});
\draw[] ({1.75*cos(300)},{1.75*sin(300)}) to ({1.5*cos(300)},{1.5*sin(300)});
\draw[] ({1.75*cos(320)},{1.75*sin(320)}) to ({1.5*cos(320)},{1.5*sin(320)});
\draw[] ({1.75*cos(340)},{1.75*sin(340)}) to ({1.5*cos(340)},{1.5*sin(340)});
\draw[] ({1.75*cos(0)},{1.75*sin(0)}) to ({1.5*cos(0)},{1.5*sin(0)});
\draw[] ({1.75*cos(20)},{1.75*sin(20)}) to ({1.5*cos(20)},{1.5*sin(20)});
\draw[] ({1.75*cos(40)},{1.75*sin(40)}) to ({1.5*cos(40)},{1.5*sin(40)});
\draw[] ({1.75*cos(60)},{1.75*sin(60)}) to ({1.5*cos(60)},{1.5*sin(60)});
\draw[] ({1.75*cos(80)},{1.75*sin(80)}) to ({1.5*cos(80)},{1.5*sin(80)});
\draw[] ({1.75*cos(100)},{1.75*sin(100)}) to ({1.5*cos(100)},{1.5*sin(100)});
\draw[] ({1.75*cos(120)},{1.75*sin(120)}) to ({1.5*cos(120)},{1.5*sin(120)});
\draw[] ({1.75*cos(140)},{1.75*sin(140)}) to ({1.5*cos(140)},{1.5*sin(140)});
\draw[] ({1.75*cos(160)},{1.75*sin(160)}) to ({1.5*cos(160)},{1.5*sin(160)});
\draw[] ({1.75*cos(180)},{1.75*sin(180)}) to ({1.5*cos(180)},{1.5*sin(180)});
\draw[] ({1.75*cos(200)},{1.75*sin(200)}) to ({1.5*cos(200)},{1.5*sin(200)});
\draw[] ({1.75*cos(220)},{1.75*sin(220)}) to ({1.5*cos(220)},{1.5*sin(220)});
\draw[] ({1.75*cos(240)},{1.75*sin(240)}) to ({1.5*cos(240)},{1.5*sin(240)});
\draw[] ({1.75*cos(260)},{1.75*sin(260)}) to ({1.5*cos(260)},{1.5*sin(260)});

\end{tikzpicture}};


 \node at (4,0) {\begin{tikzpicture}
 
\def\R{1.75}
\def\RR{1.5}

\draw[dashed,gray] (0,0) circle (\R);

\draw[line width=1.75pt, draw=Black] (0,0) ++(320:2) arc[start angle=320, end angle=480, radius=2];

\draw[line width=1.75pt, draw=Black!50] (0,0) ++(0:2.25) arc[start angle=0, end angle=300, radius=2.25];


\draw[dashed] ({1.75*cos(280)},{1.75*sin(280)}) to ({2.7*cos(280)},{2.7*sin(280)});
\node at ({2.95*cos(280)},{2.95*sin(280)}){$1$};

\draw[dashed] ({1.75*cos(260)},{1.75*sin(260)}) to ({2.7*cos(260)},{2.7*sin(260)});
\node at ({2.95*cos(260)},{2.95*sin(260)}){$n$};

\draw[dashed] ({1.75*cos(340)},{1.75*sin(340)}) to ({2.7*cos(340)},{2.7*sin(340)});
\node at ({3.15*cos(340)},{3.15*sin(340)}){$H_{AB}$};

\draw[dashed] ({1.75*cos(20)},{1.75*sin(20)}) to ({2.7*cos(20)},{2.7*sin(20)});
\node at ({3.15*cos(20)},{3.15*sin(20)}){$X'(H)$};

\draw[dashed] ({1.75*cos(100)},{1.75*sin(100)}) to ({2.7*cos(100)},{2.7*sin(100)});
\node at ({2.95*cos(100)},{2.95*sin(100)}){$\Bar{H}_{AB}$};

\draw[dashed] ({1.75*cos(180)},{1.75*sin(180)}) to ({2.7*cos(180)},{2.7*sin(180)});
\node at ({3.25*cos(180)},{3.25*sin(180)}){$X(\Bar{H})$};

\newcommand{\hwplotB}{\raisebox{2pt}{\tikz{\draw[black,dashed,line width=1.75pt](0,0) -- (4mm,0);}}}
\node at (0,0) {$H := \color{Black}{\bullet}$};
\node at (0,0.5) {$\Bar{H} := \color{Black!50}{\bullet}$};
\node at (0.18,-0.5) {$D :=$ \hwplotB};

\draw[dashed,line width=1.75pt, draw=black] (0,0) ++(-40:1.75) arc[start angle=-40, end angle=260, radius=1.75];

\draw[] ({1.75*cos(280)},{1.75*sin(280)}) to ({1.5*cos(280)},{1.5*sin(280)});
\draw[] ({1.75*cos(300)},{1.75*sin(300)}) to ({1.5*cos(300)},{1.5*sin(300)});
\draw[] ({1.75*cos(320)},{1.75*sin(320)}) to ({1.5*cos(320)},{1.5*sin(320)});
\draw[] ({1.75*cos(340)},{1.75*sin(340)}) to ({1.5*cos(340)},{1.5*sin(340)});
\draw[] ({1.75*cos(0)},{1.75*sin(0)}) to ({1.5*cos(0)},{1.5*sin(0)});
\draw[] ({1.75*cos(20)},{1.75*sin(20)}) to ({1.5*cos(20)},{1.5*sin(20)});
\draw[] ({1.75*cos(40)},{1.75*sin(40)}) to ({1.5*cos(40)},{1.5*sin(40)});
\draw[] ({1.75*cos(60)},{1.75*sin(60)}) to ({1.5*cos(60)},{1.5*sin(60)});
\draw[] ({1.75*cos(80)},{1.75*sin(80)}) to ({1.5*cos(80)},{1.5*sin(80)});
\draw[] ({1.75*cos(100)},{1.75*sin(100)}) to ({1.5*cos(100)},{1.5*sin(100)});
\draw[] ({1.75*cos(120)},{1.75*sin(120)}) to ({1.5*cos(120)},{1.5*sin(120)});
\draw[] ({1.75*cos(140)},{1.75*sin(140)}) to ({1.5*cos(140)},{1.5*sin(140)});
\draw[] ({1.75*cos(160)},{1.75*sin(160)}) to ({1.5*cos(160)},{1.5*sin(160)});
\draw[] ({1.75*cos(180)},{1.75*sin(180)}) to ({1.5*cos(180)},{1.5*sin(180)});
\draw[] ({1.75*cos(200)},{1.75*sin(200)}) to ({1.5*cos(200)},{1.5*sin(200)});
\draw[] ({1.75*cos(220)},{1.75*sin(220)}) to ({1.5*cos(220)},{1.5*sin(220)});
\draw[] ({1.75*cos(240)},{1.75*sin(240)}) to ({1.5*cos(240)},{1.5*sin(240)});
\draw[] ({1.75*cos(260)},{1.75*sin(260)}) to ({1.5*cos(260)},{1.5*sin(260)});

\end{tikzpicture}};
\end{tikzpicture}}
\caption{Examples of the first case discussed in the proof of \Cref{join}. On the left, an instance of the first subcase of \emph{Case 1} and on the right, an instance of the second subcase of \emph{Case 1}.}
\label{fig:prop5.11}\end{figure}

    For both scenarios, let $H_\cap = H_\cap|_D \in \HH$ and observe that $[X'(H),\Bar{H}_{AB}] \subseteq H_\cap \subseteq H \cap \Bar{H}$. The acyclicity of $A$ and $B$ implies that $(H_\cap)_{AB}\geq (\Bar{H})_{AB}$. Since $X'>A,B$, we obtain the following chain of inequalities:

    $$
    X'(H) < \Bar{H}_{AB} \leq (H_\cap)_{AB} \leq X'(H_\cap).
    $$

    This contradicts the acyclicity of $X'$ with a cycle on $H$ and $H_\cap$.

    \emph{Case 2}.
    Suppose that there is no hyperedge $\Bar{H}\in \HH$ such that $X'(H)\in \Bar{H}$, $X'(H)<\Bar{H}_{AB}$ and $[X'(H),\Bar{H}_{AB}] \subseteq H$. This implies that $H\in \Hc$. Let $(H_i,H_i)_{i=1}^k$ be a minimal sequence that realizes $X(H,X'(H))$. Minimality here means that $h_i\notin H_{i+2}$ for $i\in [k-2]$. This implies that $[H_i,\max(H_{i+1})] \subseteq H_{i+1}$ for each hyperedge $H_i$ where $3\leq i \leq k$. Since $X'(H_i)\geq (H_i)_{AB}$, there exists the following path in the orientation $X'$:

    $$X'(H) \xrightarrow{H}X'(H_k) \xrightarrow{H_k} X'(H_{k-1}) \xrightarrow{H_{k-1}} \cdots \xrightarrow{H_4} X'(H_3)$$ 

    Since $X'(H)<h_2 \leq X'(H_2)\leq X'(H_k)$, there exists $3\leq j \leq k$ such that $ X'(H_2)\in H_j$. Since $X'(H)\in H_2$, we find the following cycle in $X'$:

    $$X'(H) \xrightarrow{H}X'(H_k) \xrightarrow{H_k} X'(H_{k-1}) \xrightarrow{H_{k-1}} \cdots \xrightarrow{H_{j+1}} X'(H_j) \xrightarrow{H_j} X'(H_2) \xrightarrow{H_2} X'(H)$$

    This contradicts the acyclicity of $X'$. Therefore, no such acyclic $X'$ exists and the pseudo-join $X$ of $A$ and $B$ is indeed their join.
\end{proof}

\begin{theorem} \label{thmsuf}
    The hypergraphic poset $P_\HH$ is a lattice if ${(\HH|_D)}_{\reg}$ is closed under intersection and there exist a fix for every hugging quadruple of $\HH|_D$ for any interval $D = [x, y]$ where $1 \le x < y \le n$.
\end{theorem}

\begin{proof}
    \Cref{join} proves the existence of a join for any incomparable acyclic orientations $A$ and $B$ of $\HH$. If $A$ and $B$ are comparable then the larger element is the join. Since $P_\HH$ is bounded, we can conclude that $P_\HH$ is a lattice.
\end{proof}

\section{Meet and Join Description}\label{sec:join-meet-description}

In this final section, we generalize the join to more than two acyclic orientations and formalize the meet description in a similar manner.

Let $\HH$ be a cyclic interval hypergraph on $[n]$. For every hyperedge $H \in \HH$ and a set $\mathcal{S}$ of $s$ distinct acyclic orientations $A_1,\dots,A_s$ of $\HH$, let \defn{$H_{\mathcal{S}}$}$:=\max\{A \ | \ A\in \mathcal{S}\}$ and \defn{$\prescript{}{\mathcal{S}}{H}$}$:=\min\{A \ | \ A\in \mathcal{S}\}$. 

\begin{definition}
Let $H \in \HH$ and $\ell \in H$. Let $\mathcal{S}$ be a set of $s$ distinct acyclic orientations $A_i$ of $\HH$. We say that a sequence $(H_i, h_i)_{i = 1}^k$ where $H_i \in \HH$ and $h_i \in H_i$ is an \defn{$(H, \ell)$ join sequence with respect to the $\mathcal{S}$} if the following conditions are satisfied:
\begin{itemize}
    \item $H_1 = H$
    \item $h_1 = \ell$ and $h_i = (H_i)_{\mathcal{S}}$ if $i \ne 1$
    \item $h_i \in H_{i+1}$ and $h_i < h_{i+1}$ for every $i \in [k-1]$
    \item $h_k \in H$
\end{itemize}
Moreover, let 
\[X^{\mathcal{S}}(H, \ell) = \max \{h_k\ | \ (H_i, h_i)_{i = 1}^k \text{ is an $(H, \ell)$ join sequence with respect to $A$ and $B$}\}.
\]
Define 

\[
X^{\mathcal{S}}(H) = \min_{\ell \in H, \ \ell \ge H_{\mathcal{S}}} \{X^\mathcal{S}(H, \ell)\}.
\]

\end{definition}

\begin{definition}
Let $H \in \HH$ and $\ell \in H$. Let $\mathcal{S}$ be a set of $s$ distinct acyclic orientations $A_i$ of $\HH$. We say that a sequence $(H_i, h_i)_{i = 1}^k$ where $H_i \in \HH$ and $h_i \in H_i$ is an \defn{$(H, \ell)$ meet sequence with respect to the $\mathcal{S}$} if the following conditions are satisfied:
\begin{itemize}
    \item $H_1 = H$
    \item $h_1 = \ell$ and $h_i = \prescript{}{\mathcal{S}}{(H_i)}$ if $i \ne 1$
    \item $h_i \in H_{i+1}$ and $h_i > h_{i+1}$ for every $i \in [k-1]$
    \item $h_k \in H$
\end{itemize}
Moreover, let 
\[\prescript{\mathcal{S}}{}{X}(H, \ell) = \min \{h_k\ | \ (H_i, h_i)_{i = 1}^k \text{ is an $(H, \ell)$ meet sequence with respect to the $A_i$}\}.
\]
Define 

\[
\prescript{\mathcal{S}}{}{X}(H) = \max_{\ell \in H, \ \ell \ge \prescript{}{\mathcal{S}}{H}} \{\prescript{\mathcal{S}}{}{X}(H, \ell)\}.
\]

\end{definition}

\begin{proposition}
    Let $\HH$ be a cyclic interval hypergraph such that ${(\HH|_D)}_\text{reg}$ is closed under intersection and there exist a fix for every hugging quadruple of $\HH|_D$ for any interval $D=[x,y]$ where $1\leq x < y \leq n$. For any set $\mathcal{S}$ of acyclic orientation $A_1,\dots,A_s$ of $\HH$ and $H\in \HH$:

    $$\begin{matrix}
        \bigvee_{i\in [s]}A_i= X^{\mathcal{S}}(H) && \text{and} && \bigwedge_{i\in [s]}A_i= \prescript{\mathcal{S}}{}{X}(H)
    \end{matrix}$$
\end{proposition}

\begin{proof}
    The join description follows by induction using \Cref{join}. The meet description follows by symmetry.
\end{proof}

\section*{Acknowledgments}
The authors thank Nantel Bergeron and Vincent Pilaud for their detailed comments on earlier drafts, and Isabella Novik for helpful feedback. This research was partially conducted during the Intensive Research Program on Combinatorial Geometries and Geometric Combinatorics, held at the Centre de Recerca Matemàtica in October-November 2025. The authors thank Juliana Curtis and other members of the ``Exploring Hypergraphic Polytopes" project for early discussions. The Intensive Research Program was funded by the Severo Ochoa and María de Maeztu Program for Centers and Units of Excellence in R\&D (CEX2020-001084-M), the Institute de Matemàtica de la Universitat de Barcelona, and the Spanish projects PID2022-137283NB-C21 and RED2024-153572-T of MICIU/AEI /10.13039/501100011033. Yirong Yang's research was partially supported by NSF grant DMS-2246399.

\bibliographystyle{alpha}
\bibliography{bibliography}

\end{document}